\documentclass[12pt,reqno]{article}

\usepackage[usenames]{color}
\usepackage[colorlinks=true,
linkcolor=webgreen, filecolor=webbrown,
citecolor=webgreen]{hyperref}

\definecolor{webgreen}{rgb}{0,.5,0}
\definecolor{webbrown}{rgb}{.6,0,0}

\usepackage{amssymb}
\usepackage{graphicx}
\usepackage{amscd}
\usepackage{lscape}
\usepackage{tikz}
\usepackage{tikz-cd}
\usepackage{pgfplots}
\usepackage{mathtools}

\usetikzlibrary{matrix}
\usetikzlibrary{fit,shapes}
\usetikzlibrary{positioning, calc}
\tikzset{circle node/.style = {circle,inner sep=1pt,draw, fill=white},
        X node/.style = {fill=white, inner sep=1pt},
        dot node/.style = {circle, draw, inner sep=5pt}
        }
\usepackage{tkz-fct}

\usepackage{amsthm}
\newtheorem{theorem}{Theorem}

\newtheorem{proposition}[theorem]{Proposition}
\newtheorem{corollary}[theorem]{Corollary}

\theoremstyle{definition}

\newtheorem{example}[theorem]{Example}

\usepackage{float}

\usepackage{graphics,amsmath}
\usepackage{amsfonts}
\usepackage{latexsym}
\usepackage{epsf}

\DeclareMathOperator{\sech}{sech}

\newcommand{\seqnum}[1]{\href{http://oeis.org/#1}{\underline{#1}}}

\begin{document}

\begin{center}
\vskip 1cm{\LARGE\bf Three \'Etudes on a sequence transformation pipeline} \vskip 1cm \large
Paul Barry\\
School of Science\\
Waterford Institute of Technology\\
Ireland\\
\href{mailto:pbarry@wit.ie}{\tt pbarry@wit.ie}
\end{center}
\vskip .2 in

\begin{abstract} We study a sequence transformation pipeline that maps certain sequences with rational generating functions to permutation-based sequence families of combinatorial significance. Many of the number triangles we encounter can be related to simplicial objects such as the associahedron and the permutahedron. The linkages between these objects is facilitated by the use of the previously introduced $\mathcal{T}$ transform.
\end{abstract}

\section{Introduction}
In this note, we shall define a transformation pipeline beginning with certain sequences possessing a simple ordinary generating function, that combines the inverse Sumudu transform \cite{Belgacem2, Belgacem1, Watugala} (or inverse Laplace Borel transfom) with the reversion of exponential generating functions. Previously, we have studied the invertible $\mathcal{T}$ transform, that uses inversion, the Sumudu transform and reversion, beginning with an exponential generating function \cite{T}.

The sequences and the number triangles that we shall encounter will be referenced, where known, by their $Annnnnn$ number of the On-Line Encyclopedia of Integer Sequences \cite{SL1, SL2}. The comments and references to be found under the $Annnnnn$ numbers of these sequences are an invaluable aid to extending the breadth of this note. A particular theme to be gleaned in reference to many of the sequences and triangles in this note is their association with objects such as the $f$-vectors and the $h$-vectors of simplicial objects such as the associahedra, permutohedra and the stellahedra \cite{Toric}.

We adopt a number of conventions. All number triangles encountered are infinite in extent (downwards and to the right). We only exhibit short truncations of these. On occasion, we use the language of Riordan arrays \cite{Book, Survey, SGWW}. The notation $(g, f)$ signifies an ordinary Riordan array, while $[g,f]$ denotes an exponential Riordan array. Many of the sequences (including polynomial sequences) in this note are examples of moment sequences \cite{Classical, Eulerian, Barry_Moment} associated to families of orthogonal polynomials \cite{Chihara, Gautschi, Szego}. Many have generating functions expressible in continued fraction form \cite{CFT, Wall}. The notation
$$\mathcal{J}(a,b,c,\ldots; \alpha, \beta, \gamma,\ldots)$$ signifies a Jacobi-type continued fraction
$$\cfrac{1}{1- ax-
\cfrac{\alpha x^2}{1-bx-
\cfrac{\beta x^2}{1- cx-
\cfrac{\gamma x^2}{1-\cdots}}}}.$$
Similarly, the notation
$$\mathcal{S}(a,b,c,\ldots;\alpha,\beta, \gamma,\ldots)$$ signifies a Stieltjes-type continued fraction
$$\cfrac{1}{1-
\cfrac{ax}{1-
\cfrac{\alpha x}{1-
\cfrac{bx}{1-
\cfrac{\beta x}{1-\cdots}}}}}.$$ This non-conventional notation will be seen to be useful in the sequel for the patterns that will become apparent.

The Del\'eham notation
$$[r_0,r_1, r_2, \ldots] \,\Delta \, [s_0, s_1, s_2,\ldots]$$ signifies the number triangle whose bi-variate generating function is given by
$$\cfrac{1}{1-
\cfrac{r_0x+s_0xy}{1-
\cfrac{r_1x+s_1xy}{1-
\cfrac{r_2x+s_2xy}{1-\cdots}}}}.$$

The reversal of the triangle (going from $T_{n,k}$ to $T_{n,n-k}$) is then given by
$$[s_0, s_1, s_2,\ldots] \,\Delta\,[r_0,r_1, r_2, \ldots].$$

In addition, we use the notation
$$[r_0,r_1, r_2, \ldots] \,\Delta^{(1)} \, [s_0, s_1, s_2,\ldots]$$ to signify the triangle with generating function
$$\cfrac{1}{1-(r_0x+s_0xy)-
\cfrac{r_1x+s_1xy}{1-
\cfrac{r_2x+s_2xy}{1-\cdots}}}.$$

If $f(x)$ is a generating function with $f(0)\ne 0$, then by its reversion we will mean
 $$\frac{1}{x}\text{Rev}(xf(x)).$$ Operationally, we solve the equation
 $$uf(u)=x$$ and take $u(x)/x$ for the solution that satisfies $u(0)=0$.

We use the notation
$$c(x)=\frac{1-\sqrt{1-4x}}{2x}$$ to denote the generating function of the Catalan numbers $C_n=\frac{1}{n+1}\binom{2n}{n}$ \seqnum{A000108}. Thus $c(x)$ is the reversion of $1-x$ in the sense above. Also in this sense, the reversion of $\frac{1}{1+(r+1)x+rx^2}$ is $\frac{1}{1-(r+1)x}c\left(\frac{rx^2}{(1-(r+1)x)^2}\right)$. Looking at the number triangles that these generating functions expand to, we can say that the Narayana triangle $N_3$, which begins
$$\left(
\begin{array}{ccccccc}
 1 & 0 & 0 & 0 & 0 & 0 & 0 \\
 1 & 1 & 0 & 0 & 0 & 0 & 0 \\
 1 & 3 & 1 & 0 & 0 & 0 & 0 \\
 1 & 6 & 6 & 1 & 0 & 0 & 0 \\
 1 & 10 & 20 & 10 & 1 & 0 & 0 \\
 1 & 15 & 50 & 50 & 15 & 1 & 0 \\
 1 & 21 & 105 & 175 & 105 & 21 & 1 \\
\end{array}
\right),$$ is the reversion of the triangle that begins
$$\left(
\begin{array}{ccccccc}
 1 & 0 & 0 & 0 & 0 & 0 & 0 \\
 -1 & -1 & 0 & 0 & 0 & 0 & 0 \\
 1 & 1 & 1 & 0 & 0 & 0 & 0 \\
 -1 & -1 & -1 & -1 & 0 & 0 & 0 \\
 1 & 1 & 1 & 1 & 1 & 0 & 0 \\
 -1 & -1 & -1 & -1 & -1 & -1 & 0 \\
 1 & 1 & 1 & 1 & 1 & 1 & 1 \\
\end{array}
\right).$$

The ordinary generating functions $g(x)=\sum_{n=0}g_n x^n$ and the exponential generating functions $f(x)=\sum_{n=0} f_n \frac{x^n}{n!}$ that are used in this note depend on the coefficients $g_n$ and $f_n$ only. Thus $x$ is a ``dummy variable''. We variously use $x$, $z$, $t$ for this dummy variable in the note. The variable $r$ is used as a parameter, but also as such a dummy variable in bi-variate expressions. 

The Stirling numbers of the second kind, elements of the exponential Riordan array $\left[1, e^x-1\right]$ \seqnum{A048993}, will be denoted by $S(n,k)$ in this note. We have 
$$S(n,k)=\frac{1}{k!} \sum_{j=0}^k (-1)^{k-j} \binom{k}{j}j^n.$$

\section{Preliminaries: From Narayana to Euler}
There are three Narayana triangles in common usage. In terms of their general terms, these are characterised as follows.
$$N_1(n,k)=\frac{1}{k+1}\binom{n+1}{k}\binom{n}{k}.$$ This triangle begins
$$\left(
\begin{array}{ccccccc}
 1 & 0 & 0 & 0 & 0 & 0 & 0 \\
 1 & 0 & 0 & 0 & 0 & 0 & 0 \\
 1 & 1 & 0 & 0 & 0 & 0 & 0 \\
 1 & 3 & 1 & 0 & 0 & 0 & 0 \\
 1 & 6 & 6 & 1 & 0 & 0 & 0 \\
 1 & 10 & 20 & 10 & 1 & 0 & 0 \\
 1 & 15 & 50 & 50 & 15 & 1 & 0 \\
\end{array}
\right).$$
$$N_2(n,k)=\frac{1}{n-k+1} \binom{n-1}{n-k}\binom{n}{k}.$$ This triangle begins
$$\left(
\begin{array}{ccccccc}
 1 & 0 & 0 & 0 & 0 & 0 & 0 \\
 0 & 1 & 0 & 0 & 0 & 0 & 0 \\
 0 & 1 & 1 & 0 & 0 & 0 & 0 \\
 0 & 1 & 3 & 1 & 0 & 0 & 0 \\
 0 & 1 & 6 & 6 & 1 & 0 & 0 \\
 0 & 1 & 10 & 20 & 10 & 1 & 0 \\
 0 & 1 & 15 & 50 & 50 & 15 & 1 \\
\end{array}
\right).$$
Finally, we have
$$N_3(n,k)=\frac{1}{k+1} \binom{n+1}{k}\binom{n}{k}.$$ This triangle begins
$$\left(
\begin{array}{ccccccc}
 1 & 0 & 0 & 0 & 0 & 0 & 0 \\
 1 & 1 & 0 & 0 & 0 & 0 & 0 \\
 1 & 3 & 1 & 0 & 0 & 0 & 0 \\
 1 & 6 & 6 & 1 & 0 & 0 & 0 \\
 1 & 10 & 20 & 10 & 1 & 0 & 0 \\
 1 & 15 & 50 & 50 & 15 & 1 & 0 \\
 1 & 21 & 105 & 175 & 105 & 21 & 1 \\
\end{array}
\right).$$
We have
\begin{center}
\begin{tabular}{|c|c|c|c|}
\hline Triangle & A-number & Generating function & Reversion of \\ \hline
$N_1$ & \seqnum{A131198} & $\frac{1}{1+(r-1)x}c\left(\frac{rx}{(1+(r-1)x)^2}\right)$ & $\frac{1-rx}{1-(r-1)x}$ \\ \hline
$N_2$ & \seqnum{A090181} & $\frac{1}{1-(r-1)x}c\left(\frac{x}{(1-(r-1)x)^2}\right)$ & $\frac{1-x}{1+(r-1)x} $ \\ \hline
$N_3$ & \seqnum{A001263} & $\frac{1}{1-(r+1)x}c\left(\frac{rx^2}{(1-(r+1)x)^2}\right)$ & $\frac{1}{1+(r+1)x+rx^2}$ \\
\hline
\end{tabular}
\end{center}
Expressing their (ordinary) generating functions as continued fractions, we have the following.
\begin{scriptsize}
\begin{center}
\begin{tabular}{|c|c|c|c|}
\hline Type & $N_1$ & $N_2$ & $N_3$ \\ \hline
Jacobi & $\mathcal{J}(1,r+1,r+1,\ldots;r,r,r,\ldots)$ & $\mathcal{J}(r,r+1,r+1,\ldots;r,r,r,\ldots)$ & $\mathcal{J}(r+1,r+1,r+1,\ldots;r,r,r,\ldots)$ \\ \hline
Stieltjes & $\mathcal{S}(1,1,1,\ldots; r,r,r\ldots)$ & $\mathcal{S}(r,r,r,\ldots;1,1,1,\dots)$ & $---$ \\ \hline
\end{tabular}
\end{center}
\end{scriptsize}
These three triangles can be mapped to the three Eulerian triangles essentially by taking the logarithmic derivative of the inverse Sumudu (or Laplace Borel) transform of the reversion of their generating functions.
The three Eulerian triangles $E_1$, $E_2$ and $E_3$ can be characterized by their bivariate generating functions as follows.
\begin{center}
\begin{tabular}{|c|c|c|} \hline
Triangle & Generating function & A-number \\ \hline
$E_1$ & $\frac{1-r}{e^{rt}-re^t}$ & \seqnum{A173018} \\ \hline
$E_2$ & $\frac{(1-r)e^{rt}}{e^{rt}-re^t}$ & \seqnum{A123125}  \\ \hline
$E_3$ & $\frac{e^{t(r+1)}(r-1)^2}{(e^{rt}-r e^t)^2}$ & \seqnum{A008292}\\ \hline
\end{tabular}
\end{center}

For the triangle $N_1$, we start with $\frac{1-rx}{1-(r-1)x}$. We then have the following.
\begin{enumerate}
\item The inverse Sumudu transform of $\frac{1-rx}{1-(r-1)x}$ is $\frac{r-e^{t(r-1)}}{r-1}$.
\item The logarithmic derivative of this result is $\frac{(r-1)e^{t(r-1)}}{r-e^{t(r-1)}}$.
\end{enumerate}
Then the negative of this exponential generating function, or $\frac{(1-r)e^{t(r-1)}}{r-e^{t(r-1)}}$, expands to give the Eulerian triangle $E_2$
$$\left(
\begin{array}{ccccccc}
 1 & 0 & 0 & 0 & 0 & 0 & 0 \\
 0 & 1 & 0 & 0 & 0 & 0 & 0 \\
 0 & 1 & 1 & 0 & 0 & 0 & 0 \\
 0 & 1 & 4 & 1 & 0 & 0 & 0 \\
 0 & 1 & 11 & 11 & 1 & 0 & 0 \\
 0 & 1 & 26 & 66 & 26 & 1 & 0 \\
 0 & 1 & 57 & 302 & 302 & 57 & 1 \\
\end{array}
\right).$$

For the triangle $N_2$, we start with the generating function $\frac{1-x}{1-(1-r)x}$. We then have the following.
\begin{enumerate}
\item The inverse Sumudu transform of $\frac{1-x}{1-(1-r)x}$ is $\frac{re^{t(1-r)}-1}{r-1}$.
\item The logarithmic derivative of this result is $\frac{r(r-1)}{e^{t(r-1)}-r}$.
\end{enumerate}
Dividing by $r$, we get the exponential generating function $\frac{r-1}{e^{t(r-1)}-r}$ which expands to give the Eulerian triangle $E_1$
$$\left(
\begin{array}{ccccccc}
 1 & 0 & 0 & 0 & 0 & 0 & 0 \\
 1 & 0 & 0 & 0 & 0 & 0 & 0 \\
 1 & 1 & 0 & 0 & 0 & 0 & 0 \\
 1 & 4 & 1 & 0 & 0 & 0 & 0 \\
 1 & 11 & 11 & 1 & 0 & 0 & 0 \\
 1 & 26 & 66 & 26 & 1 & 0 & 0 \\
 1 & 57 & 302 & 302 & 57 & 1 & 0 \\
\end{array}
\right).$$

Finally, we have the following.
\begin{enumerate}
\item The inverse Sumudu transform of $\frac{1}{1+(r+1)x+rx^2}$ is $\frac{re^{-rt}-e{-t}}{r-1}$.
\item The logarithmic derivative of this result is $-1+\frac{e^t(r-1)}{e^{rt}-re^t}$.
\end{enumerate}
Then $$\frac{1}{r}\left(1+\left(-1+\frac{e^t(r-1)}{e^{rt}-re^t}\right)\right)=\frac{e^t(1-r)}{e^{rt}-re^t}$$
expands to give the Eulerian triangle $E_1$.
$$\left(
\begin{array}{ccccccc}
 1 & 0 & 0 & 0 & 0 & 0 & 0 \\
 1 & 0 & 0 & 0 & 0 & 0 & 0 \\
 1 & 1 & 0 & 0 & 0 & 0 & 0 \\
 1 & 4 & 1 & 0 & 0 & 0 & 0 \\
 1 & 11 & 11 & 1 & 0 & 0 & 0 \\
 1 & 26 & 66 & 26 & 1 & 0 & 0 \\
 1 & 57 & 302 & 302 & 57 & 1 & 0 \\
\end{array}
\right).$$
Alternatively, $$-\frac{1}{r}\frac{d}{dt} \left(-1+\frac{e^t(r-1)}{e^{rt}-re^t}\right)=\frac{e^{t(r+1)}(r-1)^2}{(e^{rt}-r e^t)^2}$$ expands to given the Eulerian triangle $E_3$ that begins
$$\left(
\begin{array}{ccccccc}
 1 & 0 & 0 & 0 & 0 & 0 & 0 \\
 1 & 1 & 0 & 0 & 0 & 0 & 0 \\
 1 & 4 & 1 & 0 & 0 & 0 & 0 \\
 1 & 11 & 11 & 1 & 0 & 0 & 0 \\
 1 & 26 & 66 & 26 & 1 & 0 & 0 \\
 1 & 57 & 302 & 302 & 57 & 1 & 0 \\
 1 & 120 & 1191 & 2416 & 1191 & 120 & 1 \\
\end{array}
\right).$$

We have the following table of continued fraction generating functions for the Eulerian triangles.
\begin{scriptsize}
\begin{center}
\begin{tabular}{|c|c|c|c|}
\hline Type & $E_1$ & $E_2$ & $E_3$ \\ \hline
Jacobi & $\mathcal{J}(1,r+2,2r+3,\ldots;r,4r,9r,\ldots)$ & $\mathcal{J}(r,2r+1,3r+2,\ldots;r,4r,9r,\ldots)$ & $\mathcal{J}(r+1,2(r+1),3(r+1),\ldots;2r,6r,12r,\ldots)$ \\ \hline
Stieltjes & $\mathcal{S}(1,2,3,\ldots; r,2r,3r\ldots)$ & $\mathcal{S}(r,2r,3r,\ldots;1,2,3,\dots)$ & $---$ \\ \hline
\end{tabular}
\end{center}
\end{scriptsize}
At this stage we can invoke the $\mathcal{T}$ transform \cite{T} to map $E_1$ to $N_1$ and to map $E_2$ to $N_2$.
These relationships can also be seen clearly in terms of the Del\'eham notation.
\begin{center}
\begin{tabular}{|c|c|}
\hline
$N_1$ & $[1,0,1,0,1,0,\ldots]\,\Delta\,[0,1,0,1,0,1,0,\ldots]$\\ \hline
$N_2$ & $[0,1,0,1,0,\ldots]\,\Delta\,[1,0,1,0,\ldots]$\\ \hline
$N_3$ & $[0,1,0,1,0,\ldots]\,\Delta^{(1)}\,[1,0,1,0,1,0,\ldots]$ \\ \hline
$E_1$ & $[1,0,2,0,3,0,\ldots]\,\Delta\,[0,1,0,2,0,3,0,\ldots]$\\ \hline
$E_2$ & $[0,1,0,2,0,3,0,\ldots]\,\Delta\,[1,0,2,0,3,0,\ldots]$ \\ \hline
$E_3$ &  $[0,1,0,2,0,\ldots]\,\Delta^{(1)}\,[1,0,2,0,3,0,\ldots]$ \\ \hline
\end{tabular}
\end{center}
Thus for instance the ordinary generating function of $E_3$ is given by
$$\cfrac{1}{1-xy-
\cfrac{x}{1-
\cfrac{2xy}{1-
\cfrac{2x}{1-
\cfrac{3xy}{1-\cdots}}}}}.$$

Note that the image of the symmetric Narayana triangle $N_3$ by $\mathcal{T}^{-1}$ is the triangle with generating function
$$\mathcal{J}(r+1, 2(r+1), 3(r+1),\ldots; r,4r,9r,\ldots).$$ This is \seqnum{A046802}, which begins
$$\left(
\begin{array}{ccccccc}
 1 & 0 & 0 & 0 & 0 & 0 & 0 \\
 1 & 1 & 0 & 0 & 0 & 0 & 0 \\
 1 & 3 & 1 & 0 & 0 & 0 & 0 \\
 1 & 7 & 7 & 1 & 0 & 0 & 0 \\
 1 & 15 & 33 & 15 & 1 & 0 & 0 \\
 1 & 31 & 131 & 131 & 31 & 1 & 0 \\
 1 & 63 & 473 & 883 & 473 & 63 & 1 \\
\end{array}
\right).$$
The row polynomials of this triangle are the $h$-polynomials associated to the stellahedra. Multiplying this triangle on the right by the binomial triangle $\mathbf{B}=\left(\binom{n}{k}\right)$ gives us the triangle \seqnum{A248727}, which begins
$$\left(
\begin{array}{ccccccc}
 1 & 0 & 0 & 0 & 0 & 0 & 0 \\
 2 & 1 & 0 & 0 & 0 & 0 & 0 \\
 5 & 5 & 1 & 0 & 0 & 0 & 0 \\
 16 & 24 & 10 & 1 & 0 & 0 & 0 \\
 65 & 130 & 84 & 19 & 1 & 0 & 0 \\
 326 & 815 & 720 & 265 & 36 & 1 & 0 \\
 1957 & 5871 & 6605 & 3425 & 803 & 69 & 1 \\
\end{array}
\right).$$
Its rows give the $f$-polynomials for the stellahedra \cite{Toric}.

\section{\'Etude I: An introductory example}

We introduce the transformation pipeline by way of a simple example. The rational generating function that we work with in this section is $$g(x)=\frac{1}{1-x^2}.$$ This expands to give the sequence
$$1,0,1,0,1,0,1,0,\ldots.$$
Regarded in the form
$$\frac{1}{(1-x)(1+x)},$$ its expansion is seen to give the partial sums of the sequence
$$1,-1,1,-1,1,-1,1,-1,\ldots.$$ The INVERT$(-1)$ transform of $g(x)$ is
$$ \frac{g(x)}{1-xg(x)}=\frac{1}{1-x-x^2}$$ which expands to give the Fibonacci numbers \seqnum{A000045}
$$1, 1, 2, 3, 5, 8, 13, 21, 34, 55, 89,\ldots.$$
Similarly, the INVERT$(1)$ transform of $g(x)$ is given by
$$\frac{g(x)}{1+g(x)}=\frac{1}{1+x-x^2},$$ which expands to give the signed Fibonacci numbers
$$1, -1, 2, -3, 5, -8, 13, -21, 34, -55, 89,\ldots.$$
We now wish to operate on the generating function $g(x)$ as follows.
\begin{enumerate}
\item Take the inverse Sumudu transform of $g(x)$ to get $\tilde{g}(t)$ (that is, we get the corresponding exponential generating function)
\item Take the logarithmic derivative of $\tilde{g}(t)=\frac{\tilde{g}'(t)}{\tilde{g}(t)}:=h(t)$
\item Form $1-h(t)$ and get $\int_0^z (1-h(t))\,dt$ (thus pre-pending a $0$ to the expansion of $1-h(t)$).
\item Get the derivative of the reversion of this last result.
\end{enumerate}
We shall refer to the application of this sequence of operations as the ``transformation pipeline'' $\mathcal{P}$. Note that we have chosen the sample generating function $\frac{1}{1-x^2}$ to ensure that all these steps make sense in this case, as we shall now see.
\begin{proposition} The image of the generating function $g(x)=\frac{1}{1-x^2}$ under the transformation pipeline is
$$\mathcal{P}\left(\frac{1}{1-x^2}\right)=\frac{1}{2-e^x},$$ the generating function of the Fubini numbers
$$1, 1, 3, 13, 75, 541, 4683, 47293, 545835, 7087261, \ldots\quad \seqnum{A000670}$$ which enumerate ordered partitions.
\end{proposition}
\begin{proof} The proof consists of carrying out the steps of the pipeline.
\begin{enumerate}
\item The inverse Sumudu transform of $g(x)=\frac{1}{1-x^2}$ is $\tilde{g}(t)=\cosh(t)$
\item The logarithmic derivative of $\tilde{g}(t)$ is $\tanh(t)$.
\item Calculate $\int_0^z (1-\tanh(t))\,dt=2z-\ln\left(\frac{e^{2z}}{2}+\frac{1}{2}\right)$
\item Solve $2z-\ln\left(\frac{e^{2z}}{2}+\frac{1}{2}\right)=x$ for $z$ to get $\frac{1}{2}\ln\left(2e^{-x}-1\right)$
\item Taking the derivative of this last expression gives us $\frac{1}{2-e^x}$.
\end{enumerate}
\end{proof}
A natural question that arises is whether we can extend this to the parameterized generating function $\frac{1}{1-rx^2}$. Unfortunately, this is not so easy. Taking the case of $\frac{1}{1-2x^2}$, we find that we have
$$\frac{1}{1-2x^2}\to \cosh(\sqrt{2}t)\to \sqrt{2}\tanh(\sqrt{2}t)\to (1+\sqrt{2})z-\ln\left(\frac{e^{2\sqrt{2}z}}{2}+\frac{1}{2}\right).$$
Unfortunately, the reversion of the last power series is non-elementary. We note that numerically, the
expansion of the derivative of the reversion begins
$$1, 2, 12, 112, 1440, 23648, 473088, 11164288, 303648000, 9352781312,\ldots.$$
A significant difference between this sequence and that of the Fubini numbers now emerges. The Fubini numbers admit of a continued fraction ordinary generating function, with integer coefficients. In effect, this generating function is given by
$$\cfrac{1}{1-x-
\cfrac{2x^2}{1-4x-
\cfrac{8x^2}{1-7x-
\cfrac{18x^2}{1-10x-
\cfrac{32x^2}{1-13x-\cdots}}}}},$$ or equivalently
$$\cfrac{1}{1-
\cfrac{x}{1-
\cfrac{2x}{1-
\cfrac{2x}{1-
\cfrac{4x}{1-
\cfrac{3x}{1-
\cfrac{6x}{1-\cdots}}}}}}}.$$
We use the shorthand
$$\mathcal{J}(1,4,7,10,\ldots;2,8,18,32,\ldots)$$ and
$$\mathcal{S}(1,2,2,4,3,6,4,8,5,\ldots)=\mathcal{S}(1,2,3,4,\ldots;2,4,6,8,\ldots),$$ for these continued fractions, where $\mathcal{J}$ stands for ``Jacobi'', and $\mathcal{S}$ stands for ``Stieltjes''.

We now note that the sequence $$1, 2, 12, 112, 1440, 23648, 473088, 11164288, 303648000, 9352781312,\ldots$$
does not possess an ordinary generating function expressible in terms of a continued fraction with integer coefficients.

\begin{example} We consider the sequence that begins
$$1,0,2,0,2,0,2,0,2,0,\ldots$$ with generating function
$$g(x)=\frac{1+x^2}{1-x^2}.$$
We find the following pipeline.
$$\frac{1+x^2}{1-x^2}\to 2\cosh(t)-1\to \frac{2\sinh(t)}{2\cosh(t)-1}\to 2z-\ln\left(2^{2z}-e^z+1\right).$$
Reverting this last expression and taking the derivative, we find that under the transformation pipeline,
$$\frac{1+x^2}{1-x^2} \to \frac {2} {4 - 3 e^x + \sqrt {e^x (4 - 3 e^x)}},$$ with the later
exponential generating function expanding to
$$1, 2, 12, 110, 1380, 22022, 426972, 9747950, 256176660,\ldots.$$
This suggests reversing the pipeline for sequences with generating functions of the form
$$ \frac {2} {(r+1) - r e^{ax} + \sqrt {e^{ax} ((r+1) - r e^x)}},$$ for suitable values of $a$ and $x$. A simple case  that presents itself is $a=4$, $r=0$. Then we get
$$  \frac {2} {1  + \sqrt {e^{4x} (1)}}=\frac{2}{1+e^{2x}}=\frac{e^{-z}}{\cosh(z)}=e^{-z}\sech(z),$$ which expands to give the sequence that begins
$$ 1, -1, 0, 2, 0, -16, 0, 272, 0, -7936, 0,\ldots.$$
The ordinary generating function of this sequence is
$$\mathcal{J}(-1,-1,-1,-1,\ldots;-1,-4,-9,-25,\ldots),$$ or equivalently
$$\mathcal{S}(-1,-1,-2,-2,-3,-3,-4,\ldots)=\mathcal{S}(-1,-2,-3,\ldots;-1,-2,-3,\ldots).$$
To begin the reverse pipeline, we form the sequence
$$ 0, 1, -1, 0, 2, 0, -16, 0, 272, 0, -7936, 0,\ldots$$ with exponential generating function
$$\int_0^x \frac{e^{-z}}{\cosh(z)}\,dz=2x-\ln\left(\frac{e^{2x}}{2}+\frac{1}{2}\right).$$
We now revert this sequence to get the sequence
$$ 0,1, 1, 3, 13, 75, 541, 4683, 47293, 545835, 7087261, \ldots$$ of right-shifted Fubini numbers, with generating function $-\frac{1}{2}\ln\left(2e^{-x}-1\right)$.

The logarithmic derivative sequence that we seek now begins
 $$-1, -3, -13, -75, -541, -4683, -47293, -545835, -7087261, \ldots$$ with generating function
 $$-\frac{d}{dx} \frac{1}{2-e^x}=-\frac{e^x}{(e^x-2)^2}.$$ To reverse the logarithmic derivative we integrate and take the exponential:
 $$ \int_0^z -\frac{e^x}{(e^x-2)^2}\,dx=\frac{e^z-1}{e^z-2} \to e^{\frac{e^z-1}{e^z-2}}.$$
 This is now the exponential generating function of the pre-image sequence $\tilde{g}(t)$. This sequence begins
 $$1, -1, -2, -5, -13, -12, 379, 6907, 99112, 1378941, 19514571, 284384318,\ldots.$$
 It does not have a rational ordinary generating function.

 We note that the intermediate sequence
 $$1, 3, 13, 75, 541, 4683, 47293, 545835, 7087261, \ldots$$ of once-shifted Fubini numbers with exponential generating function $\frac{e^z}{(2-e^z)^2}$ has an ordinary generating function expressible as the continued fraction
 $$\mathcal{J}(3,6,9,12,\ldots,3(n+1),\ldots;4,12,24,40,\ldots,2(n+1)(n+2),\ldots).$$
\end{example}
We have the following general result. 
\begin{proposition} Let $F(z)$ be the end of the transformation pipeline $g(x) \to F(z)$, that is, $F(x)=\mathcal{P}(g(x))$.   Then
$$\tilde{g}(t)=e^{t-\text{Rev}\left(\int_0^t F(z)\,dz\right)},$$ and
$$F(z)=\frac{d}{dz} \text{Rev}\left(\int_0^z (1-\frac{\tilde{g}'(t)}{\tilde{g}(t)}\,dt\right).$$
\end{proposition}
\begin{proof}
By the pipeline, we have
$$\int_0^z (1- \frac{\tilde{g}'(t)}{\tilde{g}(t)}\,dt=z-\ln(\tilde{g}(z))=\text{Rev}\int_0^z F(t)\,dt.$$
Thus $$\ln(\tilde{g}(z))=z-\text{Rev}\int_0^z F(t)\,dt.$$
\end{proof}
We shall now generalize the generating function $\frac{1}{1-x^2}$ to the parameterized generating function
$$g(x)=\frac{1+(r-1)x}{(1-x)(1+rx)} \to \tilde{g}(t)=\frac{e^{-rt}+r e^t}{1+r}.$$  This is the bivariate generating function of the triangle that begins
$$\left(
\begin{array}{ccccccc}
 1 & 0 & 0 & 0 & 0 & 0 & 0 \\
 0 & 0 & 0 & 0 & 0 & 0 & 0 \\
 0 & 1 & 0 & 0 & 0 & 0 & 0 \\
 0 & 1 & -1 & 0 & 0 & 0 & 0 \\
 0 & 1 & -1 & 1 & 0 & 0 & 0 \\
 0 & 1 & -1 & 1 & -1 & 0 & 0 \\
 0 & 1 & -1 & 1 & -1 & 1 & 0 \\
 0 & 1 & -1 & 1 & -1 & 1 & -1 \\
\end{array}
\right),$$
or
$$\left(\frac{1}{1-x},0 \cdot x\right)-\left(\frac{x}{1-x},-x\right)$$ in terms of Riordan arrays.

We note at this juncture that by multiplying this matrix on the right by $\mathbf{B}=\left(\binom{n}{k}\right)$ we obtain the matrix
$$\left(
\begin{array}{cccccccc}
 1 & 0 & 0 & 0 & 0 & 0 & 0 & 0 \\
 0 & 0 & 0 & 0 & 0 & 0 & 0 & 0 \\
 1 & 1 & 0 & 0 & 0 & 0 & 0 & 0 \\
 0 & -1 & -1 & 0 & 0 & 0 & 0 & 0 \\
 1 & 2 & 2 & 1 & 0 & 0 & 0 & 0 \\
 0 & -2 & -4 & -3 & -1 & 0 & 0 & 0 \\
 1 & 3 & 6 & 7 & 4 & 1 & 0 & 0 \\
 0 & -3 & -9 & -13 & -11 & -5 & -1 & 0 \\
\end{array}
\right),$$ or
$$\left(\frac{1}{1-x}, 0 \cdot x\right)-\left(\frac{x}{1-x^2}, \frac{-x}{1+x}\right),$$ which has bivariate generating function $\frac{1+rx}{(1-x)(1+(r+1)x)}$.

The sequence with generating function $g(x)=\frac{1+(r-1)x}{(1-x)(1+rx)}$ is the partial sum sequence of the sequence
$$1, -1, r, - r^2, r^3, - r^4, r^5, - r^6, r^7, - r^8, r^9,\ldots$$ with generating function $\frac{1+(r-1)x}{1+rx}$. The inverse binomial transform of $\frac{1+rx}{(1-x)(1+(r+1)x)}$ expands to give the sequence
$$1, -1, r + 1, - (r + 1)^2, (r + 1)^3, - (r + 1)^4, (r + 1)^5, - (r + 1)^6,\ldots.$$ The coefficient array of these polynomials is then
$$\left(
\begin{array}{cccccccc}
 1 & 0 & 0 & 0 & 0 & 0 & 0 & 0 \\
 -1 & 0 & 0 & 0 & 0 & 0 & 0 & 0 \\
 1 & 1 & 0 & 0 & 0 & 0 & 0 & 0 \\
 -1 & -2 & -1 & 0 & 0 & 0 & 0 & 0 \\
 1 & 3 & 3 & 1 & 0 & 0 & 0 & 0 \\
 -1 & -4 & -6 & -4 & -1 & 0 & 0 & 0 \\
 1 & 5 & 10 & 10 & 5 & 1 & 0 & 0 \\
 -1 & -6 & -15 & -20 & -15 & -6 & -1 & 0 \\
\end{array}
\right).$$ The next result shows that the reversion of this triangle is a variant of the Narayana triangle.
\begin{proposition}
The reversion of the inverse binomial transform of $g(x)=\frac{1+(r-1)x}{(1-x)(1+rx)}$ expands to the signed Narayana triangle that begins
$$\left(
\begin{array}{cccccccc}
 1 & 0 & 0 & 0 & 0 & 0 & 0 & 0 \\
 1 & 0 & 0 & 0 & 0 & 0 & 0 & 0 \\
 1 & -1 & 0 & 0 & 0 & 0 & 0 & 0 \\
 1 & -3 & 1 & 0 & 0 & 0 & 0 & 0 \\
 1 & -6 & 6 & -1 & 0 & 0 & 0 & 0 \\
 1 & -10 & 20 & -10 & 1 & 0 & 0 & 0 \\
 1 & -15 & 50 & -50 & 15 & -1 & 0 & 0 \\
 1 & -21 & 105 & -175 & 105 & -21 & 1 & 0 \\
\end{array}
\right).$$
\end{proposition}
\begin{proof}
The inverse binomial transform of $g(x)$ is given by $\frac{1+rx}{1+(r+1)x}$. Solving the equation
$$\frac{u(1+ru)}{1+(r+1)u}=x$$ for $u$ such that $u(0)=0$, we find
$$u(x)=\frac{x}{1-(r+1)x}c\left(\frac{-rx}{(1-(r+1)x)^2}\right).$$
Then $\frac{1}{1-(r+1)x}c\left(\frac{-rx}{(1-(r+1)x)^2}\right)$ expands to give the signed Narayana triangle above.
\end{proof}
After these preliminaries, we can now state the main result of this section.
\begin{proposition}$$ \mathcal{P}\left(\frac{1+(r-1)x}{(1-x)(1+rx)}\right)=\frac{1}{1+r(1-e^z)}.$$
\end{proposition}
\begin{proof}
We let $F(z)=\frac{1}{r+1-re^z}$. Then $\int_0^z F(t)\,dt=\frac{1}{r+1}(\pi i+z -\ln(r e^z-r-1))$.
Then
$$ \text{Rev}\int_0^z F(t)\,dt=(r+1)z-\ln\left(\frac{1+e^{rx(r+1)}}{r+1}\right).$$
This gives us
$$\tilde{g}(t)=e^{t-\text{Rev}\left(\int_0^t F(z)\,dz\right)}=\frac{1}{r+1} e^{-rx}(r e^{x(r+1)}+1).$$
Taking the Sumudu transform of this exponential generating function gives us
$$g(x)=\frac{1+(r-1)x}{(1-x)(1+rx)}.$$
\end{proof}
\begin{proposition} We have, for $r \ne 0$, that
$$ \mathcal{P}\left(\frac{1+(r-1)x}{(1-x)(1+rx)}\right)=\frac{1}{1+r(1-e^z)}$$ is the generating function of the moment sequence for the family of orthogonal polynomials whose coefficient array is given by the exponential Riordan array
$$\left[\frac{1}{1+rz}, \ln\left(\frac{1+(r+1)z}{1+rz}\right)\right].$$
These moments appear as the initial column elements in the inverse array
$$\left[\frac{1}{1+r(1-e^z)}, \frac{e^z-1}{1+r(1-e^z)}\right].$$
\end{proposition}
\begin{proof} We let $[g,f]=\left[\frac{1}{1+r(1-e^z)}, \frac{e^z-1}{1+r(1-e^z)}\right]$. We find that
$$A(z)=f'(\bar{f}(z))=(1+rz)(1+(r+1)z),$$ and
$$Z(x)=\frac{g'(\bar{f}(z))}{g(\bar{f}(z))}=r(1+(r+1)z).$$
Thus the production matrix of $[g,f]$ is tri-diagonal and hence $[g,f]^{-1}$ is the coefficient array of a family of orthogonal polynomials.
\end{proof}
The production matrix has generating function
$$e^{zy}(r(1+(r+1)z)+y(1+rz)(1+(r+1)z)).$$
It begins
$$\left(
\begin{array}{cccccc}
 r & 1 & 0 & 0 & 0 & 0 \\
 r (r+1) & 3 r+1 & 1 & 0 & 0 & 0 \\
 0 & 4 r (r+1) & 5 r+2 & 1 & 0 & 0 \\
 0 & 0 & 9 r (r+1) & 7 r+3 & 1 & 0 \\
 0 & 0 & 0 & 16 r (r+1) & 9 r+4 & 1 \\
 0 & 0 & 0 & 0 & 25 r (r+1) & 11 r+5 \\
\end{array}
\right).$$
\begin{corollary}
$ \mathcal{P}\left(\frac{1+(r-1)x}{(1-x)(1+rx)}\right)=\frac{1}{1+r(1-e^z)}$ is the generating function of the moments for the family of orthogonal polynomials $P_n(x;r)$ that satisfy the three-term recurrence
$$P_n(x,r)=(x-(r+(n-1)(2r+1)))P_{n-1}(x;r)-r(r+1)(n-1)^2 P_{n-2}(x;r),$$
with $P_0(x;r)=1$ and $P_1(x;r)=x-r$.
\end{corollary}
Note that the sequences with exponential generating function $\frac{1}{r+1-re^z}$ have an ordinary generating function given by
$$\mathcal{J}(r,3r+1,5r+2,\ldots;r(r+1),4r(r+1),9r(r+1),\ldots),$$ or equivalently
$$\mathcal{S}(r,2r,3r,\ldots; r+1,2(r+1),3(r+1),\ldots)=\mathcal{S}(r,r+1,2r,2(r+1),3r,3(r+1)\ldots).$$

The sequence generated by $\frac{1}{r+1-re^x}$ is the polynomial sequence that begins
$$1, r, r(2r + 1), r(6r^2 + 6r + 1), r(24r^3 + 36r^2 + 14r + 1), r(120r^4 + 240r^3 + 150r^2 + 30r + 1),\ldots,$$ with coefficient array $a_{n,k}=k! S(n,k)$, where $S(n,k)$ are the Stirling numbers of the second kind. This is \seqnum{A019538}, which begins
 $$\left(
\begin{array}{ccccccc}
 1 & 0 & 0 & 0 & 0 & 0 & 0 \\
 0 & 1 & 0 & 0 & 0 & 0 & 0 \\
 0 & 1 & 2 & 0 & 0 & 0 & 0 \\
 0 & 1 & 6 & 6 & 0 & 0 & 0 \\
 0 & 1 & 14 & 36 & 24 & 0 & 0 \\
 0 & 1 & 30 & 150 & 240 & 120 & 0 \\
 0 & 1 & 62 & 540 & 1560 & 1800 & 720 \\
\end{array}
\right).$$ This triangle counts the number of set compositions of $n$ with $k$ blocks \cite{Petersen}, among other combinatorial interpretations.  
Thus we have
$$\left(
\begin{array}{ccccccc}
 1 & 0 & 0 & 0 & 0 & 0 & 0 \\
 0 & 0 & 0 & 0 & 0 & 0 & 0 \\
 0 & 1 & 0 & 0 & 0 & 0 & 0 \\
 0 & 1 & -1 & 0 & 0 & 0 & 0 \\
 0 & 1 & -1 & 1 & 0 & 0 & 0 \\
 0 & 1 & -1 & 1 & -1 & 0 & 0 \\
 0 & 1 & -1 & 1 & -1 & 1 & 0 \\
\end{array}
\right) \xrightarrow{\mathcal{P}} \left(
\begin{array}{ccccccc}
 1 & 0 & 0 & 0 & 0 & 0 & 0 \\
 0 & 1 & 0 & 0 & 0 & 0 & 0 \\
 0 & 1 & 2 & 0 & 0 & 0 & 0 \\
 0 & 1 & 6 & 6 & 0 & 0 & 0 \\
 0 & 1 & 14 & 36 & 24 & 0 & 0 \\
 0 & 1 & 30 & 150 & 240 & 120 & 0 \\
 0 & 1 & 62 & 540 & 1560 & 1800 & 720 \\
\end{array}
\right).$$
There are many important combinatorial applications of this array, which in the Del\'eham notation is
$$[0, 1, 0, 2, 0, 3, 0, 4, 0, 5,\ldots]\, \Delta\, [1, 1, 2, 2, 3, 3, 4, 4, 5, 5, \ldots],$$ with bivariate generating function 
$$\mathcal{J}(r,3r+1,5r+2,\ldots;r(r+1),4r(r+1),9r(r+1),\ldots).$$
For $r=1\ldots 5$, this polynomial sequence $1,r,r(2r+1),\ldots$ evaluates to \seqnum{A000670}, \seqnum{A004123}, \seqnum{A032033}, \seqnum{A094417}, and \seqnum{A094418}. They are referred to as generalized ordered Bell numbers. One should see also \seqnum{A094416}, which gathers these sequences into a single array.

We remark that using the $\mathcal{T}$ transform, we can associate the sequences with generating function
$$\mathcal{J}(r,3r+1,5r+2,\ldots;r(r+1),4r(r+1),9r(r+1),\ldots)$$ with those with generating function
$$\mathcal{J}(r,2r+1,2r+1,\ldots;r(r+1),r(r+1),r(r+1),\ldots).$$ The coefficient array of the polynomial sequence defined by $\mathcal{J}(r,2r+1,2r+1,\ldots;r(r+1),r(r+1),r(r+1),\ldots)$ is \seqnum{A086810} (see also \seqnum{A033282}), or
$$ [0, 1, 0, 1, 0, 1, \ldots]\, \Delta \, [1, 1, 1, 1, 1,  \ldots].$$ This triangle, which has general element
$$\frac{1}{n+1} \binom{n-1}{n-k}\binom{n+k}{k},$$  begins
$$\left(
\begin{array}{ccccccc}
 1 & 0 & 0 & 0 & 0 & 0 & 0 \\
 0 & 1 & 0 & 0 & 0 & 0 & 0 \\
 0 & 1 & 2 & 0 & 0 & 0 & 0 \\
 0 & 1 & 5 & 5 & 0 & 0 & 0 \\
 0 & 1 & 9 & 21 & 14 & 0 & 0 \\
 0 & 1 & 14 & 56 & 84 & 42 & 0 \\
 0 & 1 & 20 & 120 & 300 & 330 & 132 \\
\end{array}
\right).$$ Thus we have
$$\left(
\begin{array}{ccccccc}
 1 & 0 & 0 & 0 & 0 & 0 & 0 \\
 0 & 0 & 0 & 0 & 0 & 0 & 0 \\
 0 & 1 & 0 & 0 & 0 & 0 & 0 \\
 0 & 1 & -1 & 0 & 0 & 0 & 0 \\
 0 & 1 & -1 & 1 & 0 & 0 & 0 \\
 0 & 1 & -1 & 1 & -1 & 0 & 0 \\
 0 & 1 & -1 & 1 & -1 & 1 & 0 \\
\end{array}
\right) \xrightarrow{\mathcal{T} \circ \mathcal{P}}
\left(
\begin{array}{ccccccc}
 1 & 0 & 0 & 0 & 0 & 0 & 0 \\
 0 & 1 & 0 & 0 & 0 & 0 & 0 \\
 0 & 1 & 2 & 0 & 0 & 0 & 0 \\
 0 & 1 & 5 & 5 & 0 & 0 & 0 \\
 0 & 1 & 9 & 21 & 14 & 0 & 0 \\
 0 & 1 & 14 & 56 & 84 & 42 & 0 \\
 0 & 1 & 20 & 120 & 300 & 330 & 132 \\
\end{array}
\right).$$
 Row $n+1$ of this triangle is the $f$-vector of the simplicial complex dual to an associahedron of type $A_n$ \cite{Fomin}.

By multiplying the expansion of $\frac{1+(r-1)x}{(1-x)(1+rx)}$ on the right by $\mathbf{B}$, we obtain the bivariate expansion of $\frac{1+rx}{(1-x)(1+(r+1)x)}$. This coefficient array begins
$$\left(
\begin{array}{cccccccc}
 1 & 0 & 0 & 0 & 0 & 0 & 0 & 0 \\
 0 & 0 & 0 & 0 & 0 & 0 & 0 & 0 \\
 1 & 1 & 0 & 0 & 0 & 0 & 0 & 0 \\
 0 & -1 & -1 & 0 & 0 & 0 & 0 & 0 \\
 1 & 2 & 2 & 1 & 0 & 0 & 0 & 0 \\
 0 & -2 & -4 & -3 & -1 & 0 & 0 & 0 \\
 1 & 3 & 6 & 7 & 4 & 1 & 0 & 0 \\
 0 & -3 & -9 & -13 & -11 & -5 & -1 & 0 \\
\end{array}
\right).$$
This can be expressed in terms of Riordan arrays as
$$\left(\frac{1}{1-x}, 0\cdot x\right)+\left(\frac{-x}{1-x^2}, \frac{-x}{1+x}\right).$$
The reversion of this triangle, with generating function $\frac{1}{1-rx}c\left(\frac{-x(r+(r+1)x}{(1-rx)^2}\right)$, begins
$$\left(
\begin{array}{cccccccc}
 1 & 0 & 0 & 0 & 0 & 0 & 0 & 0 \\
 0 & 0 & 0 & 0 & 0 & 0 & 0 & 0 \\
 -1 & -1 & 0 & 0 & 0 & 0 & 0 & 0 \\
 -1 & 0 & 1 & 0 & 0 & 0 & 0 & 0 \\
 0 & 3 & 2 & -1 & 0 & 0 & 0 & 0 \\
 1 & 3 & -4 & -5 & 1 & 0 & 0 & 0 \\
 1 & -3 & -14 & 0 & 9 & -1 & 0 & 0 \\
 0 & -9 & -4 & 35 & 15 & -14 & 1 & 0 \\
\end{array}
\right).$$

The ordinary generating function of this triangle is then
$$\mathcal{J}(0,-r,-r,-r,\ldots; -(r+1),-(r+1),-(r+1),\ldots).$$ We can associate it via the $\mathcal{T}$ transform with the triangle whose generating function is
$$\mathcal{J}(0,-r,-2r,-3r,\ldots; -(r+1),-4(r+1),-9(r+1),16(r+1),\ldots).$$ This begins
$$\left(
\begin{array}{cccccccc}
 1 & 0 & 0 & 0 & 0 & 0 & 0 & 0 \\
 0 & 0 & 0 & 0 & 0 & 0 & 0 & 0 \\
 -1 & -1 & 0 & 0 & 0 & 0 & 0 & 0 \\
 0 & 1 & 1 & 0 & 0 & 0 & 0 & 0 \\
 5 & 10 & 4 & -1 & 0 & 0 & 0 & 0 \\
 0 & -18 & -36 & -17 & 1 & 0 & 0 & 0 \\
 -61 & -183 & -136 & 33 & 46 & -1 & 0 & 0 \\
 0 & 479 & 1437 & 1329 & 263 & -107 & 1 & 0 \\
\end{array}
\right).$$ This has generating function 
$$\frac{(r+2)e^{(r+1)x}}{1+(r+1)e^{(r+2)x}}.$$ 
Multiplying on the right by $\mathbf{B}^{-1}$, we get the triangle with generating function
$$\mathcal{J}(0,-(r-1),-2(r-1),-3(r-1),\ldots; -r,-4r,-9r,-16r,\ldots),$$ which begins
$$\left(
\begin{array}{cccccccc}
 1 & 0 & 0 & 0 & 0 & 0 & 0 & 0 \\
 0 & 0 & 0 & 0 & 0 & 0 & 0 & 0 \\
 0 & -1 & 0 & 0 & 0 & 0 & 0 & 0 \\
 0 & -1 & 1 & 0 & 0 & 0 & 0 & 0 \\
 0 & -1 & 7 & -1 & 0 & 0 & 0 & 0 \\
 0 & -1 & 21 & -21 & 1 & 0 & 0 & 0 \\
 0 & -1 & 51 & -161 & 51 & -1 & 0 & 0 \\
 0 & -1 & 113 & -813 & 813 & -113 & 1 & 0 \\
\end{array}
\right).$$ The general $(n,k)$-term of this matrix is given by 
$$(-1)^k \sum_{j=0}^n \binom{n}{j}(-1)^{n-j}E_1(j,k).$$ Its has generating function is
$$\frac{(r+1)e^{rx}}{1+re^{(r+1)x}}.$$ 
This is a signed version of \seqnum{A271697}. We note that if we multiply this now on the left by $\mathbf{B}$, we get the signed Eulerian triangle
$$\left(
\begin{array}{cccccccc}
 1 & 0 & 0 & 0 & 0 & 0 & 0 & 0 \\
 1 & 0 & 0 & 0 & 0 & 0 & 0 & 0 \\
 1 & -1 & 0 & 0 & 0 & 0 & 0 & 0 \\
 1 & -4 & 1 & 0 & 0 & 0 & 0 & 0 \\
 1 & -11 & 11 & -1 & 0 & 0 & 0 & 0 \\
 1 & -26 & 66 & -26 & 1 & 0 & 0 & 0 \\
 1 & -57 & 302 & -302 & 57 & -1 & 0 & 0 \\
 1 & -120 & 1191 & -2416 & 1191 & -120 & 1 & 0 \\
\end{array}
\right),$$ with ordinary generating function
$$\mathcal{J}(1,-(r-2),-(2r-3),-(3r-4),\ldots; -r,-4r,-9r,-16r,\ldots),$$ and exponential generating function $\frac{1+r}{e^{-x(r+1)}-r}.$
The image of this under the inverse $\mathcal{T}$ transform is the signed Narayana triangle
$$\left(\begin{array}{cccccccc}
 1 & 0 & 0 & 0 & 0 & 0 & 0 & 0 \\
 1 & 0 & 0 & 0 & 0 & 0 & 0 & 0 \\
 1 & -1 & 0 & 0 & 0 & 0 & 0 & 0 \\
 1 & -3 & 1 & 0 & 0 & 0 & 0 & 0 \\
 1 & -6 & 6 & -1 & 0 & 0 & 0 & 0 \\
 1 & -10 & 20 & -10 & 1 & 0 & 0 & 0 \\
 1 & -15 & 50 & -50 & 15 & -1 & 0 & 0 \\
 1 & -21 & 105 & -175 & 105 & -21 & 1 & 0 \\
\end{array}
\right),$$ with generating function
$$\mathcal{J}(1,1-r,1-r,1-r,\ldots;-r,-r,-r,-r,\ldots),$$ or $$\frac{1}{1-(r+1)x}c\left(\frac{rx}{(1-(r+1)x)^2}\right).$$
This reverts to the generating function
$$ \frac{1+rx}{1+(r+1)x},$$ the inverse binomial transform of $g(x)$.

The inverse Sumudu transform of $\frac{1+rx}{1+(r+1)x}$ is $\frac{r+e^{-t(r+1)}}{r+1}$. The logarithmic derivative of this is
$$-\frac{1+r}{1+r e^{t(r+1)}},$$ which expands to give the variant Eulerian triangle that begins
$$\left(
\begin{array}{cccccccc}
 -1 & 0 & 0 & 0 & 0 & 0 & 0 & 0 \\
 0 & 1 & 0 & 0 & 0 & 0 & 0 & 0 \\
 0 & 1 & -1 & 0 & 0 & 0 & 0 & 0 \\
 0 & 1 & -4 & 1 & 0 & 0 & 0 & 0 \\
 0 & 1 & -11 & 11 & -1 & 0 & 0 & 0 \\
 0 & 1 & -26 & 66 & -26 & 1 & 0 & 0 \\
 0 & 1 & -57 & 302 & -302 & 57 & -1 & 0 \\
 0 & 1 & -120 & 1191 & -2416 & 1191 & -120 & 1 \\
\end{array}
\right).$$
We have the following.
$$\int_0^z \frac{1+r}{1+r e^{t(r+1)}}\,dt = z(r+1) + \ln\left(\frac{1+r}{1+re^{z(r+1)}}\right).$$
Solving the reversion equation
$$z(r+1) + \ln\left(\frac{1+r}{1+re^{z(r+1)}}\right)=x$$ we get
$$z(x)=\frac{-\ln\left(e^{-x}-r(1-e^{-x})\right)}{1+r}.$$
Then $$z'(x)=\frac{1}{1-r(e^x-1)}$$ is the bivariate generating function of the triangle that begins
$$\left(
\begin{array}{cccccccc}
 1 & 0 & 0 & 0 & 0 & 0 & 0 & 0 \\
 0 & 1 & 0 & 0 & 0 & 0 & 0 & 0 \\
 0 & 1 & 2 & 0 & 0 & 0 & 0 & 0 \\
 0 & 1 & 6 & 6 & 0 & 0 & 0 & 0 \\
 0 & 1 & 14 & 36 & 24 & 0 & 0 & 0 \\
 0 & 1 & 30 & 150 & 240 & 120 & 0 & 0 \\
 0 & 1 & 62 & 540 & 1560 & 1800 & 720 & 0 \\
 0 & 1 & 126 & 1806 & 8400 & 16800 & 15120 & 5040 \\
\end{array}
\right).$$
This is the triangle $\left(k! S(n,k)\right)$, \seqnum{A019538}.

To complete this section, we look at the generating function $\frac{1+(r-1)x}{1+rx}$. The inverse Sumudu transform gives us $\tilde{g}(t)=\frac{e^{-rt}+r-1}{r}$, whose logarithmic derivative is $-\frac{r}{1+(r-1)e^{rt}}$. Taking the negative of this, we get $\frac{r}{1+(r-1)e^{rt}}$, the bivariate generating function for the triangle that begins
$$\left(
\begin{array}{ccccccc}
 1 & 0 & 0 & 0 & 0 & 0 & 0 \\
 1 & -1 & 0 & 0 & 0 & 0 & 0 \\
 2 & -3 & 1 & 0 & 0 & 0 & 0 \\
 6 & -12 & 7 & -1 & 0 & 0 & 0 \\
 24 & -60 & 50 & -15 & 1 & 0 & 0 \\
 120 & -360 & 390 & -180 & 31 & -1 & 0 \\
 720 & -2520 & 3360 & -2100 & 602 & -63 & 1 \\
\end{array}
\right),$$ which in the Del\'eham notation is
$$[1,-1,2,-2,3,-3,\ldots]\,\Delta\,[1,0,2,0,3,0,\ldots].$$ This is a signed version of \seqnum{A130850}. It has general term $$(n-k)!(-1)^k S(n+1,n-k+1).$$
Now $$\int_0^z \frac{r}{1+(r-1)e^{rt}}\,dt = \ln\left(e^{rz}(r-1)+1\right)+\ln(r)+rz,$$ and the solution to the reversion equation
$$\ln\left(e^{rz}(r-1)+1\right)+\ln(r)+rz =x$$ is given by
$$\frac{-\ln\left(1-r+re^{-x}\right)}{r}.$$
We then have
$$\frac{d}{dx}\frac{-\ln\left(1-r+re^{-x}\right)}{r}=\frac{e^{-x}}{1-r+re^{-x}}=\frac{1}{r-(r-1)e^x}.$$
This is the bivariate generating function of the triangle that begins
$$\left(
\begin{array}{ccccccc}
 1 & 0 & 0 & 0 & 0 & 0 & 0 \\
 -1 & 1 & 0 & 0 & 0 & 0 & 0 \\
 1 & -3 & 2 & 0 & 0 & 0 & 0 \\
 -1 & 7 & -12 & 6 & 0 & 0 & 0 \\
 1 & -15 & 50 & -60 & 24 & 0 & 0 \\
 -1 & 31 & -180 & 390 & -360 & 120 & 0 \\
 1 & -63 & 602 & -2100 & 3360 & -2520 & 720 \\
\end{array}
\right),$$ or 
$$[1,0,2,0,3,0,\ldots] \,\Delta\,[1,-1,2,-2,3,-3,\ldots].$$
This is a signed version of \seqnum{A028246}, which gives the number of $k$-dimensional faces in the first barycentric subdivision of the standard $n$-dimensional simplex.

We then have
\begin{scriptsize}
$$\left(
\begin{array}{ccccccc}
 1 & 0 & 0 & 0 & 0 & 0 & 0 \\
 -1 & 1 & 0 & 0 & 0 & 0 & 0 \\
 1 & -3 & 2 & 0 & 0 & 0 & 0 \\
 -1 & 7 & -12 & 6 & 0 & 0 & 0 \\
 1 & -15 & 50 & -60 & 24 & 0 & 0 \\
 -1 & 31 & -180 & 390 & -360 & 120 & 0 \\
 1 & -63 & 602 & -2100 & 3360 & -2520 & 720 \\
\end{array}
\right)\cdot \mathbf{B}=\left(
\begin{array}{ccccccc}
 1 & 0 & 0 & 0 & 0 & 0 & 0 \\
 0 & 1 & 0 & 0 & 0 & 0 & 0 \\
 0 & 1 & 2 & 0 & 0 & 0 & 0 \\
 0 & 1 & 6 & 6 & 0 & 0 & 0 \\
 0 & 1 & 14 & 36 & 24 & 0 & 0 \\
 0 & 1 & 30 & 150 & 240 & 120 & 0 \\
 0 & 1 & 62 & 540 & 1560 & 1800 & 720 \\
\end{array}
\right),$$
\end{scriptsize}
where $\mathbf{B}$ is the binomial matrix $\left(\binom{n}{k}\right)$. The action of $\mathbf{B}$ in this case is on the second parameter $r$ with the effect $r \to r+1$. Thus we have
$$\frac{1}{r-(r-1)e^x} \to \frac{1}{1+r(1-e^x)}.$$
The generating function $(1-x)g(x)=\frac{1+(r-1)x}{1+rx}$ is the bivariate generating function of the triangle that begins
$$\left(
\begin{array}{ccccccc}
 1 & 0 & 0 & 0 & 0 & 0 & 0 \\
 -1 & 0 & 0 & 0 & 0 & 0 & 0 \\
 0 & 1 & 0 & 0 & 0 & 0 & 0 \\
 0 & 0 & -1 & 0 & 0 & 0 & 0 \\
 0 & 0 & 0 & 1 & 0 & 0 & 0 \\
 0 & 0 & 0 & 0 & -1 & 0 & 0 \\
 0 & 0 & 0 & 0 & 0 & 1 & 0 \\
\end{array}
\right).$$
We calculate its reversion, that is, the coefficient array whose bivariate generating function is
$$\frac{1}{x} \text{Rev} \frac{1+(r-1)x}{1+rx}=\frac{1}{1-rx}c\left(\frac{(1-r)x}{(1-rx)^2}\right).$$
This gives us the triangle that begins
$$\left(
\begin{array}{cccccc}
 1 & 0 & 0 & 0 & 0 & 0 \\
 1 & 0 & 0 & 0 & 0 & 0 \\
 2 & -1 & 0 & 0 & 0 & 0 \\
 5 & -5 & 1 & 0 & 0 & 0 \\
 14 & -21 & 9 & -1 & 0 & 0 \\
 42 & -84 & 56 & -14 & 1 & 0 \\
\end{array}
\right).$$
The generating function for this triangle is then
$$\mathcal{J}(1,2-r,2-r,2-r,\ldots; 1-r,1-r,1-r,\ldots),$$ or
$$[1,-1,1,-1,\ldots]\,\Delta\,[0,1,0,1,0,\ldots].$$ It is a signed version of \seqnum{A126216}. Under the $\mathcal{T}^{-1}$ map this is transformed into
$$\mathcal{J}(1,3-r,5-2r,7-3r,\ldots; 1-r,4(1-r),9(1-r),\ldots).$$ This expands to give the triangle that begins
$$\left(
\begin{array}{ccccccc}
 1 & 0 & 0 & 0 & 0 & 0 & 0 \\
 1 & 0 & 0 & 0 & 0 & 0 & 0 \\
 2 & -1 & 0 & 0 & 0 & 0 & 0 \\
 6 & -6 & 1 & 0 & 0 & 0 & 0 \\
 24 & -36 & 14 & -1 & 0 & 0 & 0 \\
 120 & -240 & 150 & -30 & 1 & 0 & 0 \\
 720 & -1800 & 1560 & -540 & 62 & -1 & 0 \\
\end{array}
\right).$$
The general term of this array is $(n-k)!(-1)^k  S(n,n-k)$. This is a signed version of \seqnum{A090582} or
$$[1, 1, 2, 2, 3, 3, \ldots]\, \Delta \, [0, 1, 0, 2, 0, 3, 0, \ldots],$$
with generating function
$$\frac{r}{e^{-rx}+r-1}.$$
We can operate on each of these ``on the right'' by the binomial matrix to get an equivalent sequence of related matrices. Effectively, we change $r$ to $r+1$ in each of the generating functions.
Thus we start with $\frac{1+rx}{1+(r+1)x}$, which expands to the triangle
\begin{scriptsize}
$$\left(
\begin{array}{ccccccc}
 1 & 0 & 0 & 0 & 0 & 0 & 0 \\
 -1 & 0 & 0 & 0 & 0 & 0 & 0 \\
 0 & 1 & 0 & 0 & 0 & 0 & 0 \\
 0 & 0 & -1 & 0 & 0 & 0 & 0 \\
 0 & 0 & 0 & 1 & 0 & 0 & 0 \\
 0 & 0 & 0 & 0 & -1 & 0 & 0 \\
 0 & 0 & 0 & 0 & 0 & 1 & 0 \\
\end{array}
\right)\cdot \mathbf{B}=\left(
\begin{array}{ccccccc}
 1 & 0 & 0 & 0 & 0 & 0 & 0 \\
 -1 & 0 & 0 & 0 & 0 & 0 & 0 \\
 1 & 1 & 0 & 0 & 0 & 0 & 0 \\
 -1 & -2 & -1 & 0 & 0 & 0 & 0 \\
 1 & 3 & 3 & 1 & 0 & 0 & 0 \\
 -1 & -4 & -6 & -4 & -1 & 0 & 0 \\
 1 & 5 & 10 & 10 & 5 & 1 & 0 \\
\end{array}
\right).$$
\end{scriptsize}
We now have the reversion of triangles
\begin{scriptsize}
$$\left(\begin{array}{ccccccc}
 1 & 0 & 0 & 0 & 0 & 0 & 0 \\
 -1 & 0 & 0 & 0 & 0 & 0 & 0 \\
 1 & 1 & 0 & 0 & 0 & 0 & 0 \\
 -1 & -2 & -1 & 0 & 0 & 0 & 0 \\
 1 & 3 & 3 & 1 & 0 & 0 & 0 \\
 -1 & -4 & -6 & -4 & -1 & 0 & 0 \\
 1 & 5 & 10 & 10 & 5 & 1 & 0 \\
\end{array}
\right) \xrightarrow{\text{revert}} \left(
\begin{array}{ccccccc}
 1 & 0 & 0 & 0 & 0 & 0 & 0 \\
 1 & 0 & 0 & 0 & 0 & 0 & 0 \\
 1 & -1 & 0 & 0 & 0 & 0 & 0 \\
 1 & -3 & 1 & 0 & 0 & 0 & 0 \\
 1 & -6 & 6 & -1 & 0 & 0 & 0 \\
 1 & -10 & 20 & -10 & 1 & 0 & 0 \\
 1 & -15 & 50 & -50 & 15 & -1 & 0 \\
\end{array}
\right),$$
\end{scriptsize} where we have
\begin{scriptsize}
$$\left(
\begin{array}{ccccccc}
 1 & 0 & 0 & 0 & 0 & 0 & 0 \\
 1 & 0 & 0 & 0 & 0 & 0 & 0 \\
 2 & -1 & 0 & 0 & 0 & 0 & 0 \\
 5 & -5 & 1 & 0 & 0 & 0 & 0 \\
 14 & -21 & 9 & -1 & 0 & 0 & 0 \\
 42 & -84 & 56 & -14 & 1 & 0 & 0 \\
 132 & -330 & 300 & -120 & 20 & -1 & 0 \\
\end{array}
\right)\cdot \mathbf{B}=\left(
\begin{array}{ccccccc}
 1 & 0 & 0 & 0 & 0 & 0 & 0 \\
 1 & 0 & 0 & 0 & 0 & 0 & 0 \\
 1 & -1 & 0 & 0 & 0 & 0 & 0 \\
 1 & -3 & 1 & 0 & 0 & 0 & 0 \\
 1 & -6 & 6 & -1 & 0 & 0 & 0 \\
 1 & -10 & 20 & -10 & 1 & 0 & 0 \\
 1 & -15 & 50 & -50 & 15 & -1 & 0 \\
\end{array}
\right).$$
\end{scriptsize}
This is a signed version of the Narayana triangle. The generating function of this triangle is
$$\mathcal{J}(1,1-r,1-r,1-r,\ldots; -r,-r,-r,\ldots).$$ Under the $\mathcal{T}^{-1}$ transform, this is mapped to
$$\mathcal{J}(1,2-r,3-2r,4-3r,\ldots;-r,-4r,-9r,\ldots),$$ which expands to give the signed Eulerian triangle
\begin{scriptsize}
$$\left(\begin{array}{ccccccc}
 1 & 0 & 0 & 0 & 0 & 0 & 0 \\
 1 & 0 & 0 & 0 & 0 & 0 & 0 \\
 1 & -1 & 0 & 0 & 0 & 0 & 0 \\
 1 & -4 & 1 & 0 & 0 & 0 & 0 \\
 1 & -11 & 11 & -1 & 0 & 0 & 0 \\
 1 & -26 & 66 & -26 & 1 & 0 & 0 \\
 1 & -57 & 302 & -302 & 57 & -1 & 0 \\
\end{array}
\right)=\left(
\begin{array}{ccccccc}
 1 & 0 & 0 & 0 & 0 & 0 & 0 \\
 1 & 0 & 0 & 0 & 0 & 0 & 0 \\
 2 & -1 & 0 & 0 & 0 & 0 & 0 \\
 6 & -6 & 1 & 0 & 0 & 0 & 0 \\
 24 & -36 & 14 & -1 & 0 & 0 & 0 \\
 120 & -240 & 150 & -30 & 1 & 0 & 0 \\
 720 & -1800 & 1560 & -540 & 62 & -1 & 0 \\
\end{array}
\right)\cdot \mathbf{B}.$$
\end{scriptsize}
This then has the generating function $$\frac{r+1}{e^{-(r+1)x}+r}.$$
The following observation is appropriate. If we reverse the two coefficient arrays
\begin{scriptsize}
$$\left(
\begin{array}{ccccccc}
 1 & 0 & 0 & 0 & 0 & 0 & 0 \\
 0 & 1 & 0 & 0 & 0 & 0 & 0 \\
 0 & 1 & 2 & 0 & 0 & 0 & 0 \\
 0 & 1 & 6 & 6 & 0 & 0 & 0 \\
 0 & 1 & 14 & 36 & 24 & 0 & 0 \\
 0 & 1 & 30 & 150 & 240 & 120 & 0 \\
 0 & 1 & 62 & 540 & 1560 & 1800 & 720 \\
\end{array}
\right) \xrightarrow{\mathcal{T}^{-1}} \left(
\begin{array}{ccccccc}
 1 & 0 & 0 & 0 & 0 & 0 & 0 \\
 0 & 1 & 0 & 0 & 0 & 0 & 0 \\
 0 & 1 & 2 & 0 & 0 & 0 & 0 \\
 0 & 1 & 5 & 5 & 0 & 0 & 0 \\
 0 & 1 & 9 & 21 & 14 & 0 & 0 \\
 0 & 1 & 14 & 56 & 84 & 42 & 0 \\
 0 & 1 & 20 & 120 & 300 & 330 & 132 \\
\end{array}
\right), $$
\end{scriptsize} and multiply each on the right by the inverse of the binomial matrix, we obtain respectively the Euler triangle $E_1$ and the Narayana triangle $N_1$. As shown in \cite{T}, the two triangles $E_1$ and $N_1$ are paired triangles under the $\mathcal{T}$ transform.
\begin{scriptsize}
$$\left(
\begin{array}{ccccccc}
 1 & 0 & 0 & 0 & 0 & 0 & 0 \\
 1 & 0 & 0 & 0 & 0 & 0 & 0 \\
 2 & 1 & 0 & 0 & 0 & 0 & 0 \\
 6 & 6 & 1 & 0 & 0 & 0 & 0 \\
 24 & 36 & 14 & 1 & 0 & 0 & 0 \\
 120 & 240 & 150 & 30 & 1 & 0 & 0 \\
 720 & 1800 & 1560 & 540 & 62 & 1 & 0 \\
\end{array}
\right)\cdot \mathbf{B}^{-1}=\left(
\begin{array}{ccccccc}
 1 & 0 & 0 & 0 & 0 & 0 & 0 \\
 1 & 0 & 0 & 0 & 0 & 0 & 0 \\
 1 & 1 & 0 & 0 & 0 & 0 & 0 \\
 1 & 4 & 1 & 0 & 0 & 0 & 0 \\
 1 & 11 & 11 & 1 & 0 & 0 & 0 \\
 1 & 26 & 66 & 26 & 1 & 0 & 0 \\
 1 & 57 & 302 & 302 & 57 & 1 & 0 \\
\end{array}
\right).$$
$$\left(
\begin{array}{ccccccc}
 1 & 0 & 0 & 0 & 0 & 0 & 0 \\
 1 & 0 & 0 & 0 & 0 & 0 & 0 \\
 2 & 1 & 0 & 0 & 0 & 0 & 0 \\
 5 & 5 & 1 & 0 & 0 & 0 & 0 \\
 14 & 21 & 9 & 1 & 0 & 0 & 0 \\
 42 & 84 & 56 & 14 & 1 & 0 & 0 \\
 132 & 330 & 300 & 120 & 20 & 1 & 0 \\
\end{array}
\right)\cdot \mathbf{B}^{-1}=\left(
\begin{array}{ccccccc}
 1 & 0 & 0 & 0 & 0 & 0 & 0 \\
 1 & 0 & 0 & 0 & 0 & 0 & 0 \\
 1 & 1 & 0 & 0 & 0 & 0 & 0 \\
 1 & 3 & 1 & 0 & 0 & 0 & 0 \\
 1 & 6 & 6 & 1 & 0 & 0 & 0 \\
 1 & 10 & 20 & 10 & 1 & 0 & 0 \\
 1 & 15 & 50 & 50 & 15 & 1 & 0 \\
\end{array}
\right).$$
\end{scriptsize}
\begin{example} In this example, we review how, starting from a number triangle with a simple rational bivariate generating function, we can associate to it, in a reversible manner, other triangles. By the partial $\mathcal{P}$ transform, we shall understand the two steps: taking the inverse Sumudu transform, followed by taking the logarithmic derivative of this. For this example we shall start with the generating function $G(x)=\frac{1-(r+1)x}{(1-x)(1-rx)}=\frac{1-(r+1)x}{1-(r+1)x+rx^2}$. This expands to give the triangle that begins
$$\left(
\begin{array}{ccccccc}
 1 & 0 & 0 & 0 & 0 & 0 & 0 \\
 0 & 0 & 0 & 0 & 0 & 0 & 0 \\
 0 & -1 & 0 & 0 & 0 & 0 & 0 \\
 0 & -1 & -1 & 0 & 0 & 0 & 0 \\
 0 & -1 & -1 & -1 & 0 & 0 & 0 \\
 0 & -1 & -1 & -1 & -1 & 0 & 0 \\
 0 & -1 & -1 & -1 & -1 & -1 & 0 \\
\end{array}
\right).$$ 
The reversion of $G(x)$ is $\frac{1}{1+(r+1)x}c\left(\frac{x(1+r+rx)}{(1+(r+1)x)^2}\right)$ which expands to give the triangle that begins
$$\left(
\begin{array}{ccccccc}
 1 & 0 & 0 & 0 & 0 & 0 & 0 \\
 0 & 0 & 0 & 0 & 0 & 0 & 0 \\
 0 & 1 & 0 & 0 & 0 & 0 & 0 \\
 0 & 1 & 1 & 0 & 0 & 0 & 0 \\
 0 & 1 & 4 & 1 & 0 & 0 & 0 \\
 0 & 1 & 8 & 8 & 1 & 0 & 0 \\
 0 & 1 & 13 & 29 & 13 & 1 & 0 \\
\end{array}
\right).$$ This is essentially \seqnum{A100754}, \cite{Athanasiadis}. Its row sums are the Fine numbers \seqnum{A00957} \cite{Fine}. Its generating function may be represented as the continued fraction 
$$\mathcal{J}(0,r+1,r+1,\ldots; r,r,r,\ldots).$$ 
We now have 
\begin{proposition} $$\mathcal{P}\left(\frac{1-(r+1)x}{(1-x)(1-rx)}\right)=\frac{1}{1+r(e^x-1)}.$$
\end{proposition}
\begin{proof} The inverse Sumudu transform of $G(x)$ is $\tilde{G}(t)=\frac{r e^t-e^{rt}}{r-1}$. 
The logarithmic derivative of $\tilde{G}(t)$ is $\frac{r(e^t-e^{rt})}{re^t-e^{rt}}$. 
We now form $$1-\frac{r(e^t-e^{rt})}{re^t-e^{rt}}=\frac{e^{rt}(r-1)}{re^t-e^{rt}},$$ which is the generating function of the Euler triangle $E_2$ 
$$\left(
\begin{array}{ccccccc}
 1 & 0 & 0 & 0 & 0 & 0 & 0 \\
 0 & 1 & 0 & 0 & 0 & 0 & 0 \\
 0 & 1 & 1 & 0 & 0 & 0 & 0 \\
 0 & 1 & 4 & 1 & 0 & 0 & 0 \\
 0 & 1 & 11 & 11 & 1 & 0 & 0 \\
 0 & 1 & 26 & 66 & 26 & 1 & 0 \\
 0 & 1 & 57 & 302 & 302 & 57 & 1 \\
\end{array}
\right).$$ 
We then solve the equation
$$\int_0^z \frac{e^{rt}(r-1)}{re^t-e^{rt}}\,dt=x$$ to get $$z(x)=\frac{\ln(r-e^{-x}(r-1))}{r-1}.$$ 
Finally, we differentiate this last result to get $\frac{1}{1+r(e^x-1)}$. 
\end{proof}
The generating function $\mathcal{P}\left(\frac{1-(r+1)x}{(1-x)(1-rx)}\right)=\frac{1}{1+r(e^x-1)}$ expands to give the triangle that begins 
$$\left(
\begin{array}{ccccccc}
 1 & 0 & 0 & 0 & 0 & 0 & 0 \\
 0 & -1 & 0 & 0 & 0 & 0 & 0 \\
 0 & -1 & 2 & 0 & 0 & 0 & 0 \\
 0 & -1 & 6 & -6 & 0 & 0 & 0 \\
 0 & -1 & 14 & -36 & 24 & 0 & 0 \\
 0 & -1 & 30 & -150 & 240 & -120 & 0 \\
 0 & -1 & 62 & -540 & 1560 & -1800 & 720 \\
\end{array}
\right).$$ 
The ordinary generating function for this triangle takes the form of the continued fraction 
$$\mathcal{J}(-r,1-3r,2-5r,3-7r,\ldots; r(r-1),4r(r-1),16r(r-1),\ldots).$$ 
Now the $\mathcal{T}$ transform maps this triangle to the triangle given by 
$$\mathcal{J}(-r,1-2r,1-2r,1-2r,\ldots; r(r-1), r(r-1),r(r-1),\ldots).$$ 
This triangle begins 
$$\left(
\begin{array}{ccccccc}
 1 & 0 & 0 & 0 & 0 & 0 & 0 \\
 0 & -1 & 0 & 0 & 0 & 0 & 0 \\
 0 & -1 & 2 & 0 & 0 & 0 & 0 \\
 0 & -1 & 5 & -5 & 0 & 0 & 0 \\
 0 & -1 & 9 & -21 & 14 & 0 & 0 \\
 0 & -1 & 14 & -56 & 84 & -42 & 0 \\
 0 & -1 & 20 & -120 & 300 & -330 & 132 \\
\end{array}
\right).$$ 
Thus we have associated the initial triangle 
$$\left(
\begin{array}{ccccccc}
 1 & 0 & 0 & 0 & 0 & 0 & 0 \\
 0 & 0 & 0 & 0 & 0 & 0 & 0 \\
 0 & -1 & 0 & 0 & 0 & 0 & 0 \\
 0 & -1 & -1 & 0 & 0 & 0 & 0 \\
 0 & -1 & -1 & -1 & 0 & 0 & 0 \\
 0 & -1 & -1 & -1 & -1 & 0 & 0 \\
 0 & -1 & -1 & -1 & -1 & -1 & 0 \\
\end{array}
\right)$$ in a number of reversible ways (reversion, the $\mathcal{P}$ transform, the $\mathcal{T} \circ \mathcal{P}$ transform, and a partial $\mathcal{P}$ transform) to four other triangles, each with rich combinatorial interpretations.
\end{example}

\begin{example} For this example, our starting point is $g(x)=\frac{1+(r-1)x}{(1-x)(1+rx)}=\frac{1+(r-1)x}{1+(r-1)x-rx^2}$. This expands to give
$$1, 0, r, r(1 - r), r(r^2 - r + 1), r(1 - r)(r^2 + 1), r(r^4 - r^3 + r^2 - r + 1),\ldots,$$ with 
coefficient array that begins
$$\left(
\begin{array}{ccccccc}
 1 & 0 & 0 & 0 & 0 & 0 & 0 \\
 0 & 0 & 0 & 0 & 0 & 0 & 0 \\
 0 & 1 & 0 & 0 & 0 & 0 & 0 \\
 0 & 1 & -1 & 0 & 0 & 0 & 0 \\
 0 & 1 & -1 & 1 & 0 & 0 & 0 \\
 0 & 1 & -1 & 1 & -1 & 0 & 0 \\
 0 & 1 & -1 & 1 & -1 & 1 & 0 \\
\end{array}
\right).$$ 
The reversion of this triangle, with generating function 
$$\frac{1}{1-x(r-1)}c\left(\frac{x(1-r(1+x))}{(1-x(r-1))^2}\right),$$ begins 
$$\left(
\begin{array}{ccccccc}
 1 & 0 & 0 & 0 & 0 & 0 & 0 \\
 0 & 0 & 0 & 0 & 0 & 0 & 0 \\
 0 & -1 & 0 & 0 & 0 & 0 & 0 \\
 0 & -1 & 1 & 0 & 0 & 0 & 0 \\
 0 & -1 & 4 & -1 & 0 & 0 & 0 \\
 0 & -1 & 8 & -8 & 1 & 0 & 0 \\
 0 & -1 & 13 & -29 & 13 & -1 & 0 \\
\end{array}
\right).$$ 
Invoking the $\mathcal{P}$ pipeline, we have 
\begin{enumerate}
\item $\tilde{g}(t)=\frac{e^{-rt}+re^t}{r+1}$.
\item The logarithmic derivative of $\tilde{g}(t)$ is $1-\frac{r+1}{1+re^{t(r+1)}}$. 
\item We have $1-(1-\frac{r+1}{1+re^{t(r+1)}})= \frac{r+1}{1+re^{t(r+1)}}$. 
\item We have $\int_0^z \frac{r+1}{1+re^{t(r+1)}}\,dt=z(r+1)+\ln(r+1)-\ln\left(1+re^{z(r+1)}\right)$. 
\item Solving $z(r+1)+\ln(r+1)-\ln\left(1+re^{z(r+1)}\right)=x$ and differentiating gives us $\frac{1}{1+r(1-e^x)}$. 
\end{enumerate}
Thus we have 
\begin{proposition} $$\mathcal{P}\left(\frac{1+(r-1)x}{1+(r-1)x-rx^2}\right)=\frac{1}{1+r(1-e^x)}.$$
\end{proposition}
The generating function $\frac{1}{1+r(1-e^x)}$ expands to give 

$$\left(
\begin{array}{ccccccc}
 1 & 0 & 0 & 0 & 0 & 0 & 0 \\
 0 & 1 & 0 & 0 & 0 & 0 & 0 \\
 0 & 1 & 2 & 0 & 0 & 0 & 0 \\
 0 & 1 & 6 & 6 & 0 & 0 & 0 \\
 0 & 1 & 14 & 36 & 24 & 0 & 0 \\
 0 & 1 & 30 & 150 & 240 & 120 & 0 \\
 0 & 1 & 62 & 540 & 1560 & 1800 & 720 \\
\end{array}
\right).$$ 
Note that at the intermediate stage, the generating function $\frac{r+1}{1+r e^{t(r+1)}}$ expands to give the signed Eulerian triangle that begins
$$\left(
\begin{array}{ccccccc}
 1 & 0 & 0 & 0 & 0 & 0 & 0 \\
 0 & -1 & 0 & 0 & 0 & 0 & 0 \\
 0 & -1 & 1 & 0 & 0 & 0 & 0 \\
 0 & -1 & 4 & -1 & 0 & 0 & 0 \\
 0 & -1 & 11 & -11 & 1 & 0 & 0 \\
 0 & -1 & 26 & -66 & 26 & -1 & 0 \\
 0 & -1 & 57 & -302 & 302 & -57 & 1 \\
\end{array}
\right).$$

We note now that letting $r \to r+1$ brings us from $\frac{1+(r-1)}{1+(r-1)x-rx^2}$ to $\frac{1+rx}{1+rx-(r+1)x^2}$. Then 
$$\mathcal{P}\left(\frac{1+rx}{1+rx-(r+1)x^2}\right)=\frac{1}{1+(r+1)(1-e^x)}.$$ 
This last generating function expands to give 
$$\left(
\begin{array}{ccccccc}
 1 & 0 & 0 & 0 & 0 & 0 & 0 \\
 1 & 1 & 0 & 0 & 0 & 0 & 0 \\
 3 & 5 & 2 & 0 & 0 & 0 & 0 \\
 13 & 31 & 24 & 6 & 0 & 0 & 0 \\
 75 & 233 & 266 & 132 & 24 & 0 & 0 \\
 541 & 2071 & 3120 & 2310 & 840 & 120 & 0 \\
 4683 & 21305 & 39842 & 39180 & 21360 & 6120 & 720 \\
\end{array}
\right),$$ which is 
$$\left(
\begin{array}{ccccccc}
 1 & 0 & 0 & 0 & 0 & 0 & 0 \\
 0 & 1 & 0 & 0 & 0 & 0 & 0 \\
 0 & 1 & 2 & 0 & 0 & 0 & 0 \\
 0 & 1 & 6 & 6 & 0 & 0 & 0 \\
 0 & 1 & 14 & 36 & 24 & 0 & 0 \\
 0 & 1 & 30 & 150 & 240 & 120 & 0 \\
 0 & 1 & 62 & 540 & 1560 & 1800 & 720 \\
\end{array}
\right)\cdot \mathbf{B}.$$ 
Note that again we see that $\mathcal{P}\left(\frac{1}{1-x^2}\right)=\frac{1}{2-e^x}$ by letting $r=0$. 
Setting $r=1$ shows that $\mathcal{P}\left(\frac{1+x}{1+x-2x^2}\right)=\frac{1}{3-2x^2}$. Thus the sequence \seqnum{A151575} which begins 
$$1, 0, 2, -2, 6, -10, 22, -42, 86, -170, 342,\ldots$$ is mapped by $\mathcal{P}$ to the sequence \seqnum{A004123} which begins 
$$1, 2, 10, 74, 730, 9002, 133210,\ldots.$$ The sequence \seqnum{A151575} is a signed version of \seqnum{A078008}, which counts closed walks starting and ending at the same vertex of a triangle. The sequence \seqnum{A004123} is associated to generalizations of the permutahedron \cite{Santocanale}.
\end{example}
\begin{example} The examples of $g(x)$ that we have worked with so far have expanded to sequences that begin $1,0,\ldots$, and the pipeline $\mathcal{P}$ has been designed to work with this in mind. In the following example, our initial $g(x)$ does not follow this pattern, so we modify the pipeline appropriately.

We take $g(x)=\frac{1-(r+1)x}{1-x}$. This expands to give the sequence 
$$1, -r, -r, -r, -r, -r, -r,\ldots.$$ The reversion of $g(x)$ is given by 
$$\frac{1}{1+x}c\left(\frac{(r+1)x}{(1+x)^2}\right).$$ This expands to give the number triangle that begins
$$\left(
\begin{array}{ccccccc}
 1 & 0 & 0 & 0 & 0 & 0 & 0 \\
 0 & 1 & 0 & 0 & 0 & 0 & 0 \\
 0 & 1 & 2 & 0 & 0 & 0 & 0 \\
 0 & 1 & 5 & 5 & 0 & 0 & 0 \\
 0 & 1 & 9 & 21 & 14 & 0 & 0 \\
 0 & 1 & 14 & 56 & 84 & 42 & 0 \\
 0 & 1 & 20 & 120 & 300 & 330 & 132 \\
\end{array}
\right).$$ 
As we have seen, row $n+1$ of this triangle is the $f$-vector of the simplicial complex dual to an associahedron of type $A_n$. 

Taking the inverse Sumudu transform of $g(x)$, we get the exponential generating function $\tilde{g}(t)=1+r(1-e^t)$, whose logarithmic derivative is
$$\frac{re^t}{re^t-r-1}.$$ 
We now form 
$$1- \int_0^z \frac{re^t}{re^t-r-1}\,dt=1+i \pi-\ln\left(r(e^x-1)-1\right).$$
This expands to give the number triangle that begins 
$$\left(
\begin{array}{ccccccc}
 1 & 0 & 0 & 0 & 0 & 0 & 0 \\
 0 & 1 & 0 & 0 & 0 & 0 & 0 \\
 0 & 1 & 1 & 0 & 0 & 0 & 0 \\
 0 & 1 & 3 & 2 & 0 & 0 & 0 \\
 0 & 1 & 7 & 12 & 6 & 0 & 0 \\
 0 & 1 & 15 & 50 & 60 & 24 & 0 \\
 0 & 1 & 31 & 180 & 390 & 360 & 120 \\
\end{array}
\right).$$ 
This is an extended form of \seqnum{A028246}, whose $(n,k)$-element gives the number of $k$-dimensional faces in the first barycentric subdivision of the standard $n$-dimensional simplex \cite{Brenti}.
The reversion of this triangle begins 
$$\left(
\begin{array}{cccccccc}
 1 & 0 & 0 & 0 & 0 & 0 & 0 & 0 \\
 0 & -1 & 0 & 0 & 0 & 0 & 0 & 0 \\
 0 & -1 & 2 & 0 & 0 & 0 & 0 & 0 \\
 0 & -1 & 7 & -7 & 0 & 0 & 0 & 0 \\
 0 & -1 & 18 & -52 & 34 & 0 & 0 & 0 \\
 0 & -1 & 41 & -253 & 437 & -213 & 0 & 0 \\
 0 & -1 & 88 & -1020 & 3453 & -4203 & 1630 & 0 \\
 0 & -1 & 183 & -3707 & 21670 & -49044 & 45783 & -14747 \\
\end{array}
\right).$$
The sequence on the diagonal is an alternating sign version of \seqnum{A074059}, which gives the dimensions of the cohomology ring of the moduli space of $n$-pointed curves of genus $0$ satisfying the associativity equations of physics. If we reverse this triangle and then get the reversion of the resulting triangle, we get the triangle that begins 
$$\left(
\begin{array}{ccccccc}
 1 & 0 & 0 & 0 & 0 & 0 & 0 \\
 1 & 0 & 0 & 0 & 0 & 0 & 0 \\
 1 & 1 & 0 & 0 & 0 & 0 & 0 \\
 2 & 3 & 1 & 0 & 0 & 0 & 0 \\
 6 & 12 & 7 & 1 & 0 & 0 & 0 \\
 24 & 60 & 50 & 15 & 1 & 0 & 0 \\
 120 & 360 & 390 & 180 & 31 & 1 & 0 \\
\end{array}
\right).$$ This has generating function 
$$1+ \ln\left(\frac{r}{r+1-e^{rx}}\right).$$ 
Beheading this array gives the array 
$$\left(
\begin{array}{ccccccc}
 1 & 0 & 0 & 0 & 0 & 0 & 0 \\
 1 & 1 & 0 & 0 & 0 & 0 & 0 \\
 2 & 3 & 1 & 0 & 0 & 0 & 0 \\
 6 & 12 & 7 & 1 & 0 & 0 & 0 \\
 24 & 60 & 50 & 15 & 1 & 0 & 0 \\
 120 & 360 & 390 & 180 & 31 & 1 & 0 \\
 720 & 2520 & 3360 & 2100 & 602 & 63 & 1 \\
\end{array}
\right),$$ which is \seqnum{A130850}. This is 
$$[1,1,2,2,3,3,\ldots]\, \Delta\, [1,0,2,0,3,0,\ldots],$$ with general term 
$$\sum_{i=0}^{n-k} (-1)^{n-i-k} \binom{n-k}{i}(i+1)^n,$$ and exponential generating function 
$$\frac{r}{(r+1)e^{-rx}-1}.$$ Its ordinary generating function is
$$\mathcal{J}(r+1,2r+3,3r+5,\ldots;r+1,4(r+1),9(r+1),\ldots).$$ 
The $\mathcal{T}$ transform thus maps it to 
$$\mathcal{J}(r+1,r+2,r+2,\ldots;r+1,r+1,r+1,\ldots),$$ which expands to give the triangle that begins
$$\left(
\begin{array}{ccccccc}
 1 & 0 & 0 & 0 & 0 & 0 & 0 \\
 1 & 1 & 0 & 0 & 0 & 0 & 0 \\
 2 & 3 & 1 & 0 & 0 & 0 & 0 \\
 5 & 10 & 6 & 1 & 0 & 0 & 0 \\
 14 & 35 & 30 & 10 & 1 & 0 & 0 \\
 42 & 126 & 140 & 70 & 15 & 1 & 0 \\
 132 & 462 & 630 & 420 & 140 & 21 & 1 \\
\end{array}
\right).$$ This is \seqnum{A060693}, whose $(n,k)$-element counts Schr\"oder paths of length $2n$ with $k$ peaks \cite{Novelli}. This is $N_2 \cdot \mathbf{B}$, which is equal to
$$[1,1,1,1,1,1,\ldots] \,\Delta\,[1,0,1,0,1,0,\ldots].$$ 
\end{example}
\section{\'Etude II}
In this section, we apply the pipeline $\mathcal{P}$ defined above to the family of sequences with ordinary generating function $$g(x)=\frac{1-2x}{1-2x-rx^2}=1-\frac{rx^2}{1-2x-rx^2}=\left(1, \frac{x^2}{1-2x}\right)\cdot \frac{1}{1-rx}.$$ Thus $g(x)=g(x,r)$ expands to give the sequence
$$a_n(r)=\sum_{k=0}^{\lfloor \frac{n}{2} \rfloor} \binom{n-k-1}{n-2k}2^{n-2k}r^k,$$ which begins
$$1, 0, r, 2r, r^2 + 4r, 4r^2 + 8r, r^3 + 12r^2 + 16r,\ldots$$ with coefficient array that begins
$$\left(
\begin{array}{ccccccc}
 1 & 0 & 0 & 0 & 0 & 0 & 0 \\
 0 & 0 & 0 & 0 & 0 & 0 & 0 \\
 0 & 1 & 0 & 0 & 0 & 0 & 0 \\
 0 & 2 & 0 & 0 & 0 & 0 & 0 \\
 0 & 4 & 1 & 0 & 0 & 0 & 0 \\
 0 & 8 & 4 & 0 & 0 & 0 & 0 \\
 0 & 16 & 12 & 1 & 0 & 0 & 0 \\
\end{array}
\right).$$
For $r=0\dots 3$ these sequences are, respectively
$$1, 0, 0, 0, 0, 0, 0, 0, 0, 0, 0,\ldots,$$
$$1, 0, 1, 2, 5, 12, 29, 70, 169, 408, 985,\ldots,$$
$$1, 0, 2, 4, 12, 32, 88, 240, 656, 1792, 4896,\ldots$$
$$1, 0, 3, 6, 21, 60, 183, 546, 1641, 4920, 14763,\ldots.$$
The second sequence is a variant of the Pell numbers, while the last sequence \seqnum{A054878} counts the number of closed walks of length $n$ along the edges of a tetrahedron based at a vertex.
The inverse binomial transform of $a_n(r)$ has generating function $\frac{1-x}{1-(r+1)x^2}$ and begins
$$1, -1, r + 1, -(r + 1), (r + 1)^2, - (r + 1)^2, (r + 1)^3, - (r + 1)^3, (r + 1)^4, - (r + 1)^4, (r + 1)^5,\ldots.$$
The INVERT$(-1)$ transform of $g(x)$, that is, $\frac{g(x)}{1-xg(x)}$, expands to give the sequence that begins
$$1, 1, r + 1, 4r + 1, r^2 + 11r + 1, 7r^2 + 26r + 1, r^3 + 30r^2 + 57r + 1, 10r^3 + 102r^2 + 120r + 1,\ldots.$$
For $r=0\ldots3$ we get the sequence
$$1, 1, 1, 1, 1, 1, 1, 1, 1, 1, 1,\dots,$$
$$1, 1, 2, 5, 13, 34, 89, 233, 610, 1597, 4181,\ldots,\quad \seqnum{A001519}$$
$$1, 1, 3, 9, 27, 81, 243, 729, 2187, 6561, 19683,\ldots, \quad \seqnum{A133494}$$ and
$$1, 1, 4, 13, 43, 142, 469, 1549, 5116, 16897, 55807,\ldots,\quad \seqnum{A003688}$$
For instance, for $r=1$ we obtain the sequence \seqnum{A001519}, essentially a bisection of the Fibonacci numbers.
The reversion of the above triangle has generating function
$$\frac{1}{1+2x}c\left(\frac{x(2-rx)}{(1+2x)^2}\right).$$ This expands to give the triangle that begins
$$\left(
\begin{array}{ccccccc}
 1 & 0 & 0 & 0 & 0 & 0 & 0 \\
 0 & 0 & 0 & 0 & 0 & 0 & 0 \\
 0 & -1 & 0 & 0 & 0 & 0 & 0 \\
 0 & -2 & 0 & 0 & 0 & 0 & 0 \\
 0 & -4 & 2 & 0 & 0 & 0 & 0 \\
 0 & -8 & 10 & 0 & 0 & 0 & 0 \\
 0 & -16 & 36 & -5 & 0 & 0 & 0 \\
\end{array}
\right).$$ The above generating function is equal to
$$\mathcal{J}(0,2,2,2,\ldots;-r,-r,-r,\ldots).$$ This triangle is a stretched version of the triangle that begins
$$\left(
\begin{array}{ccccccc}
 1 & 0 & 0 & 0 & 0 & 0 & 0 \\
 0 & -1 & 0 & 0 & 0 & 0 & 0 \\
 0 & -2 & 2 & 0 & 0 & 0 & 0 \\
 0 & -4 & 10 & -5 & 0 & 0 & 0 \\
 0 & -8 & 36 & -42 & 14 & 0 & 0 \\
 0 & -16 & 112 & -224 & 168 & -42 & 0 \\
 0 & -32 & 320 & -960 & 1200 & -660 & 132 \\
\end{array}
\right).$$ This is a signed scaled version of \seqnum{A086810}. It has general term
$$\frac{(-1)^k 2^{n-k}}{n+1}\binom{n-1}{n-k}\binom{n+k}{k}.$$  Its generating function is given by
$$\mathcal{S}(-r,-r,-r,\ldots;2-r,2-r,2-r,\ldots),$$ or equivalently
$$\mathcal{J}(-r,2(1-r),2(1-r),\ldots;r(r-2),r(r-2),r(r-2),\ldots).$$
\begin{proposition} We have $$\mathcal{P}(g(x))=\frac{1}{\sqrt{1+r(1- e^{2x})}}.$$
\end{proposition}
\begin{proof}
We calculate
$$\tilde{g}(t)=e^{t-\text{Rev}\left(\int_0^t \frac{1}{\sqrt{1+r-r e^{2x}}}\,dx\right)}.$$
We have
$$=\int_0^t \frac{1}{\sqrt{1+r-r e^{2x}}}\,dx=\frac{1}{\sqrt{r+1}}\left(\ln(\sqrt{-re^{2t}+r+1}-\sqrt{r+1})-\ln(1-\sqrt{r+1})-t\right).$$
We now solve for $x=x(t)$ in
$$\frac{1}{\sqrt{r+1}}\left(\ln(\sqrt{-re^{2x}+r+1}-\sqrt{r+1})-\ln(1-\sqrt{r+1})-x\right)=t.$$
We then arrive at
$$\tilde{g}(t)=e^{t-x(t)}=e^t \left(\cosh(\sqrt{1+r}t)-\frac{\sinh(\sqrt{1+r}t)}{\sqrt{1+r}}\right).$$
Taking the Sumudu transform of $\tilde{g}(t)$ now gives $g(x)=\frac{1-2x}{1-2x-rx^2}$ as required.
\end{proof}
The generating function $\frac{1}{\sqrt{1+r-r e^{2x}}}$ expands to give the Galton-type triangle \cite{Galton} that begins
$$\left(
\begin{array}{cccccccc}
 1 & 0 & 0 & 0 & 0 & 0 & 0 & 0 \\
 0 & 1 & 0 & 0 & 0 & 0 & 0 & 0 \\
 0 & 2 & 3 & 0 & 0 & 0 & 0 & 0 \\
 0 & 4 & 18 & 15 & 0 & 0 & 0 & 0 \\
 0 & 8 & 84 & 180 & 105 & 0 & 0 & 0 \\
 0 & 16 & 360 & 1500 & 2100 & 945 & 0 & 0 \\
 0 & 32 & 1488 & 10800 & 27300 & 28350 & 10395 & 0 \\
 0 & 64 & 6048 & 72240 & 294000 & 529200 & 436590 & 135135 \\
\end{array}
\right).$$ This is \seqnum{A211402}.
Its reversion begins
$$\left(
\begin{array}{ccccccc}
 1 & 0 & 0 & 0 & 0 & 0 & 0 \\
 0 & -1 & 0 & 0 & 0 & 0 & 0 \\
 0 & -2 & 0 & 0 & 0 & 0 & 0 \\
 0 & -4 & 2 & 0 & 0 & 0 & 0 \\
 0 & -8 & 16 & 0 & 0 & 0 & 0 \\
 0 & -16 & 88 & -16 & 0 & 0 & 0 \\
 0 & -32 & 416 & -272 & 0 & 0 & 0 \\
\end{array}
\right).$$
The ordinary generating function of the above Galton triangle is
$$\mathcal{S}(r,3r,5r,\ldots; 2(r+1),4(r+1),6(r+1)\ldots),$$ or equivalently
$$\mathcal{J}(r, 5r+2, 9r+4, 13r+6,\ldots; 1\cdot 2r(r+1), 3\cdot4r(r+1), 5\cdot 6r(r+1),7\cdot 8r(r+1),\ldots).$$ In the Del\'eham notation, it is
$$[0,2,0,4,0,6,0,\ldots] \, \Delta \, [1,2,3,4,5,\ldots].$$

Note that the triangle

$$[0,2,0,4,0,6,0,\ldots] \, \Delta \, [1,0,1,0,1,\ldots]$$ is the exponential Riordan array 
$$\left[1, \frac{1}{2} (e^{2x}-1)\right]$$ of generalized Stirling numbers $S2(n,k)$ of the second kind. This is \seqnum{A075497}. The above Galton array then has general element $(2k-1)!!S2(n,k)$.

\begin{proposition}
We have, for $r\ne 0$, that
$$\mathcal{P}\left(\frac{1-2x}{1-2x-rx^2}\right)=\frac{1}{1+r(1-e^{2x})}$$ is the moment sequence for the family of orthogonal polynomials whose coefficient array is given by the exponential Riordan array
$$\left[\frac{1}{1+2rx}, \frac{1}{2}\ln\left(\frac{1+2(r+1)x}{1+2rx}\right)\right].$$
These moments appear as the initial column elements in the inverse array
$$\left[\frac{1}{1+r(1-e^{2x})},\frac{e^{2x}-1}{2(1+r(1-e^{2x}))}\right].$$
\end{proposition}
\begin{proof}. We let $[g,f]=\left[\frac{1}{1+r(1-e^{2x})},\frac{e^{2x}-1}{2(1+r(1-e^{2x}))}\right]$. We find that
$$A(x)=f'(\bar{f}(x))=(1+2rx)(1+2(r+1)x),$$ and
$$Z(x)=\frac{g'(\bar{f}(x))}{g(\bar{f}(x))} = r(1+2(r+1)x).$$
Thus the production matrix of $[g,f]$ is tri-diagonal and $[g,f]^{-1}$ is the coefficient array of a family of orthogonal polynomials.
\end{proof}
The production matrix has generating function
$$ e^{xy}(r(1+2(r+1)x)+y(1+2rx)(1+2(r+1)x)).$$
It begins
$$\left(
\begin{array}{cccccc}
 r & 1 & 0 & 0 & 0 & 0 \\
 2 r (r+1) & 5 r+2 & 1 & 0 & 0 & 0 \\
 0 & 12 r (r+1) & 9 r+4 & 1 & 0 & 0 \\
 0 & 0 & 30 r (r+1) & 13 r+6 & 1 & 0 \\
 0 & 0 & 0 & 56 r (r+1) & 17 r+8 & 1 \\
 0 & 0 & 0 & 0 & 90 r (r+1) & 21 r+10 \\
\end{array}
\right).$$
\begin{corollary} $\mathcal{P}\left(\frac{1-2x}{1-2x-rx^2}\right)=\frac{1}{1+r(1-e^{2x})}$ is the generating function of the moments of the family of orthogonal polynomials $P_n(x;r)$ that satisfy the three-term recurrence
$$P_n(x;r)=(x-(r+(n-1)(4r+2)))P_{n-1}(x;r)-(n-1)(2n-3)2r(r+1)P_{n-2}(x;r),$$
with $P_0(x;r)=1$, $P_1(x;r)=x-r$.
\end{corollary}
We note that for $r=0$, we get the moment matrix $\left[1, \frac{1}{2}(e^{2x}-1)\right]$ \seqnum{A075497} of scaled Stirling numbers of the second kind.

\section{\'Etude III}
For this section, we consider the generating function
$$\frac{1-3x-(r-2)x^2}{(1-x)(1-2r-2rx^2)}=1+x^2 \frac{r(1-2x)}{(1-x)(1-2x-2rx^2)}.$$ This expands to give a sequence that begins
$$1, 0, r, r, r(2r + 1), r(6r + 1), r(4r^2 + 14r + 1), r(20r^2 + 30r + 1), \ldots.$$
For $r=1$, this gives the sequence
$$ 1, 0, 1, 1, 3, 7, 19, 51, 139, 379, 1035,\ldots,$$ where the sequence \seqnum{A052948} which begins
$$ 1, 1, 3, 7, 19, 51, 139, 379, 1035,\ldots. $$ This is related to the descent polytopes \cite{descent}.
The inverse binomial transform gives the sequence that begins
$$1, -1, r + 1, - (2r+1), (r + 1)(2r + 1), - (2r + 1)^2, (r + 1)(2r + 1)^2, - (2r + 1)^3, (r + 1)·(2·r + 1)^3,\ldots.$$ For $r=1$, we get the sequence that begins
$$1, -1, 2, -3, 6, -9, 18, -27, 54, -81, 162,\ldots.$$ the absolute value of this sequence \seqnum{A038754} counts all paths of length $n$, starting at the initial node on the path graph $P_5$. It is an eigen-sequence of the matrix that begins
$$\left(
\begin{array}{cccccccc}
 1 & 0 & 0 & 0 & 0 & 0 & 0 & 0 \\
 1 & 1 & 0 & 0 & 0 & 0 & 0 & 0 \\
 1 & 0 & 1 & 0 & 0 & 0 & 0 & 0 \\
 1 & 0 & 1 & 1 & 0 & 0 & 0 & 0 \\
 1 & 0 & 1 & 0 & 1 & 0 & 0 & 0 \\
 1 & 0 & 1 & 0 & 1 & 1 & 0 & 0 \\
 1 & 0 & 1 & 0 & 1 & 0 & 1 & 0 \\
 1 & 0 & 1 & 0 & 1 & 0 & 1 & 1 \\
\end{array}
\right).$$
The image of $ 1, -1, 2, -3, 6, -9, 18, -27,\ldots$ by this matrix is the sequence
$$1, 0, 3, 0, 9, 0, 27, 0, 81, 0, 243,\ldots.$$

\begin{proposition} We have
$$\mathcal{P}\left(\frac{1-3x-(r-2)x^2}{(1-x)(1-2r-2rx^2)}\right)= \frac{1}{\sqrt{1+2r(1-e^z)}}.$$
\end{proposition}
The coefficient array for the polynomial family defined by $\frac{1}{\sqrt{1+2r(1-e^z)}}$ begins
$$\left(
\begin{array}{cccccccc}
 1 & 0 & 0 & 0 & 0 & 0 & 0 & 0 \\
 0 & 1 & 0 & 0 & 0 & 0 & 0 & 0 \\
 0 & 1 & 3 & 0 & 0 & 0 & 0 & 0 \\
 0 & 1 & 9 & 15 & 0 & 0 & 0 & 0 \\
 0 & 1 & 21 & 90 & 105 & 0 & 0 & 0 \\
 0 & 1 & 45 & 375 & 1050 & 945 & 0 & 0 \\
 0 & 1 & 93 & 1350 & 6825 & 14175 & 10395 & 0 \\
 0 & 1 & 189 & 4515 & 36750 & 132300 & 218295 & 135135 \\
\end{array}
\right).$$
It is related to the Galton matrix of the previous section by its $(n,k)$-term being that of the former divided by $2^{n-k}$. The ordinary generating function is given by
$$\mathcal{S}(r,3r,5r,\ldots; 2r+1, 2(2r+1), 3(2r+1),\ldots),$$ or equivalently,
$$\mathcal{J}(r, 5r+1, 9r+2, \ldots; r(2r+1), 2\cdot 3 r(2r+1), 3\cdot 5 r(2r+1),\ldots).$$
In the Del\'eham notation, the triangle is given by
$$[0, 1, 0, 2, 0, 3, \ldots]\, \Delta\, [1, 2, 3, 4,5, \ldots].$$
This is \seqnum{A211608}.
Reverting
$$\int_0^z \frac{1}{\sqrt{1+2r(1-e^t)}}\,dt,$$ we obtain the generating function
$$ -\frac{\sqrt{2 r+1} \left(\left(-r+\sqrt{2 r+1}-1\right) e^{-\sqrt{2 r+1} x}+r\right)}{\left(-r+\sqrt{2
   r+1}-1\right) e^{-\sqrt{2 r+1} x}-r} $$
of the signed Andr\'e triangle, which begins
$$\left(
\begin{array}{ccccccccc}
 1 & 0 & 0 & 0 & 0 & 0 & 0 & 0 & 0 \\
 0 & 1 & 0 & 0 & 0 & 0 & 0 & 0 & 0 \\
 0 & -1 & 0 & 0 & 0 & 0 & 0 & 0 & 0 \\
 0 & 1 & -1 & 0 & 0 & 0 & 0 & 0 & 0 \\
 0 & -1 & 4 & 0 & 0 & 0 & 0 & 0 & 0 \\
 0 & 1 & -11 & 4 & 0 & 0 & 0 & 0 & 0 \\
 0 & -1 & 26 & -34 & 0 & 0 & 0 & 0 & 0 \\
 0 & 1 & -57 & 180 & -34 & 0 & 0 & 0 & 0 \\
 0 & -1 & 120 & -768 & 496 & 0 & 0 & 0 & 0 \\
\end{array}
\right).$$ Note that the absolute value of this triangle (essentially \seqnum{A094503}) then has generating function
$$\frac{\sqrt{1-2 r} \left(r-\left(r+\sqrt{1-2 r}-1\right) e^{\sqrt{1-2 r} x}\right)}{\left(r+\sqrt{1-2
   r}-1\right) e^{\sqrt{1-2 r} x}+r}.$$
The ``unstretched'' version of this is the triangle \seqnum{A096078} that begins
$$\left(
\begin{array}{ccccccc}
 1 & 0 & 0 & 0 & 0 & 0 & 0 \\
 1 & 1 & 0 & 0 & 0 & 0 & 0 \\
 1 & 4 & 4 & 0 & 0 & 0 & 0 \\
 1 & 11 & 34 & 34 & 0 & 0 & 0 \\
 1 & 26 & 180 & 496 & 496 & 0 & 0 \\
 1 & 57 & 768 & 4288 & 11056 & 11056 & 0 \\
 1 & 120 & 2904 & 28768 & 141584 & 349504 & 349504 \\
\end{array}
\right).$$
This is defined by
$$T_{n,k}=(k+1)T_{n-1,k}+(n-k+1)T_{n,k-1}.$$ The diagonal elements are the reduced tangent numbers \seqnum{A002105}.

\begin{proposition} We have, for $r \ne 0$, that
$$\mathcal{P}\left(\frac{1-3x-(r-2)x^2}{(1-x)(1-2x-2rx^2)}\right)=\frac{1}{\sqrt{1+2r(1-e^z)}}$$ is the generating function of the  moment sequence for the family of orthogonal polynomials whose coefficient array is given by the exponential Riordan array
$$\left[\frac{1}{\sqrt{1+2rz}}, \ln\left(\frac{1+z(1+2r)}{1+2rz}\right)\right].$$ These moments appear as the initial column in the inverse array
$$\left[ \frac{1}{\sqrt{1+2r(1-e^z)}},\frac{e^z-1}{1-2r(1-e^z)}\right].$$
\end{proposition}
\begin{proof} Let $[g, f]=\left[ \frac{1}{\sqrt{1+2r(1-e^z)}},\frac{e^z-1}{1-2r(1-e^z)}\right]$.
We find that $$A(z)=f'(\bar{f}(z))=(1+2rx)(1+(2r+1)z),$$ and
$$Z(z)=\frac{g'(\bar{f}(z))}{g(\bar{f}(z))}=r(1+(1+2r)z).$$
Thus the production matrix of $[g,f]$ is tri-diagonal and hence $[g,f]^{-1}$ is the coefficient array of a family of orthogonal polynomials.
\end{proof}
The production matrix has generating function
$$e^{zy}(r(1+(1+2r)z)+y((1+2rz)(1+(2r+1)z))).$$ It begins
$$\left(
\begin{array}{ccccc}
 r & 1 & 0 & 0 & 0 \\
 r (2 r+1) & 5 r+1 & 1 & 0 & 0 \\
 0 & 6r(2r+1) & 9 r+2 & 1 & 0 \\
 0 & 0 & 15r(2r+1) & 13 r+3 & 1 \\
 0 & 0 & 0 & 28r(2r+1) & 17r+4 \\
\end{array}
\right).$$
\begin{corollary} $\mathcal{P}\left(\frac{1-3x-(r-2)x^2}{(1-x)(1-2x-2rx^2)}\right)=\frac{1}{\sqrt{1+2r(1-e^x)}}$ is the generating function of the moment sequence for the family of ortohgonal polynomials $P_n(x;r)$ that satisfy the three-term recurrence
$$P_n(x;r)=(x-(r+(n-1)(4r+1))P_{n-1}(x;r)-(n-1)(r(2r+1)+(n-2)2r(2r+1))P_{n-2}(x;r),$$ with
$P_0(x;r)=1$ and $P_1(x;r)=x-r$.
\end{corollary}
We note that the sequence generated by $\frac{1}{\sqrt{1+2rz}}$ which begins
$$1, -r, 3r^2, - 15r^3, 105r^4, - 945r^5, 10395r^6,\ldots$$ is the moment sequence for the orthogonal polynomials with coefficient array
$$\left[\frac{1}{\sqrt{1-2rz}}, \frac{z}{1-2rz}\right]=\left[\frac{1}{\sqrt{1+2rz}}, \frac{z}{1+2rz}\right]^{-1}.$$
See \seqnum{A176230}.

\section{Acknowledgements} This note makes reference to many sequences to be found in the OEIS, which at the time of writing contains more than $300,000$ sequences. All who work in the area of integer sequences are profoundly indebted to Neil Sloane. 

Many of the sequences in this note are related to simplicial objects such as the associahedron and the permutahedron. Indeed, the $\mathcal{T}$ transform provides an enumerative link between these two objects, while the $\mathcal{P}$ pipeline brings these two objects back to more basic objects. 

The comments of Tom Copeland \cite{blog} and Peter Bala in the relevant OEIS entries have been very useful in this context.

\bigskip
\hrule
\bigskip
\noindent 2010 {\it Mathematics Subject Classification}: Primary
11B83; Secondary 33C45, 42C05, 15B36, 11C20, 05A15, 05E45, 	44A10
\noindent \emph{Keywords:} Sumudu transform, series reversion, Narayana triangle, Eulerian triangle, Riordan array, orthogonal polynomial, associahedron, permutahedron, stellahedron.

\bigskip
\hrule
\bigskip
\noindent (Concerned with sequences
\seqnum{A000045},
\seqnum{A000108},
\seqnum{A000670},
\seqnum{A000957},
\seqnum{A001263},
\seqnum{A001519},
\seqnum{A002105},
\seqnum{A003688},
\seqnum{A004123},
\seqnum{A008292},
\seqnum{A019538},
\seqnum{A028246},
\seqnum{A032033},
\seqnum{A033282},
\seqnum{A038754},
\seqnum{A046802},
\seqnum{A048993},
\seqnum{A052948},
\seqnum{A060693},
\seqnum{A074059},
\seqnum{A075497},
\seqnum{A078008},
\seqnum{A086810},
\seqnum{A090181},
\seqnum{A090582},
\seqnum{A094416},
\seqnum{A094417},
\seqnum{A094418},
\seqnum{A094503},
\seqnum{A096078},
\seqnum{A100754},
\seqnum{A123125},
\seqnum{A126216},
\seqnum{A130850},
\seqnum{A131198},
\seqnum{A133494},
\seqnum{A151575},
\seqnum{A173018},
\seqnum{A176230},
\seqnum{A211402},
\seqnum{A211608},
\seqnum{A248727}, and
\seqnum{A271697}.)

\end{document}